\documentclass[11pt]{amsart}

\usepackage{amssymb,amsfonts,psfrag,amscd,enumerate,epsfig,latexsym,amsmath}
\usepackage{pinlabel}

\newcommand{\TT}[1]{{{\tt #1}}}
\renewcommand{\TT}[1]{}
\usepackage{xspace}

\newtheorem{thm}{Theorem}
\newtheorem{lem}[thm]{Lemma}

\theoremstyle{definition}

\newtheorem{defn}[thm]{Definition}
\newtheorem{defns}[thm]{Definitions}

\newtheorem{example}[thm]{Example}

\theoremstyle{remark}

\newtheorem{rmk}[thm]{Remark}

\newtheorem{rmks}[thm]{Remarks}

\newcommand{\hd}{{\hat{d}}}

\newcommand{\hrho}{{\hat{\rho}}}

\newcommand{\tgamma}{{\tilde{\gamma}}}

\newcommand{\tpsi}{{\tilde{\psi}}}

\newcommand{\ttheta}{{\tilde{\theta}}}

\newcommand{\tQ}{{\widetilde{Q}}}

\newcommand{\te}{{\tilde{e}}}

\newcommand{\tp}{{\tilde{p}}}
\newcommand{\tq}{{\tilde{q}}}
\newcommand{\tv}{{\tilde{v}}}

\newcommand{\tx}{{\tilde{x}}}

\newcommand{\br}{{\bar r}}
\newcommand{\bh}{{\bar h}}

\newcommand{\bcV}{{\overline\cV}}

\newcommand{\brho}{{\bar{\rho}}}

\newcommand{\csph}{{\widehat{\C}}}

\newcommand{\I}{^{-1}}
\newcommand{\co}{\colon}

\newcommand{\llangle}{\left\langle}

\newcommand{\rrangle}{\right\rangle}

\newcommand{\sbs}{\subset}

\newcommand{\ol}[1]{\overline{#1}}
\newcommand{\interior}[1]{{\stackrel{\circ}{#1}}}
\newcommand{\opna}[1]{\operatorname{#1}}

\newcommand{\segpair}[2]{\left\langle #1,#2\right\rangle}
\newcommand{\ssegpair}[1]{\segpair{#1}{#1'}}

\newcommand{\meas}[1]{\rmm_G\left(#1\right)}

\newcommand{\er}[2]{[#1]_{#2}}
\newcommand{\erk}[1]{\er{#1}{k}}

\newcommand{\qr}{(q;r)}
\newcommand{\qt}{(q;t)}
\newcommand{\lr}{(\Lambda;r)}
\newcommand{\ql}{(q,\Lambda)}
\newcommand{\cmqr}{\cm\qr}
\newcommand{\cnqr}{\cn\qr}
\newcommand{\nccr}{\ncc(r)}
\newcommand{\NCL}{\NC(\Lambda)}
\newcommand{\PRq}{\PR(q)}
\newcommand{\CLqr}{\CL\qr}
\newcommand{\CBqr}{\CB\qr}
\newcommand{\CCqr}{\CC\qr}
\newcommand{\qrs}{(q;r,s)}
\newcommand{\qrt}{(q;r,t)}
\newcommand{\qs}{(q;s)}

\newcommand{\rmd}{\,\mathrm{d}}
\newcommand{\rmm}{\mathrm{m}}

\newcommand{\rmR}{{\mathrm{R}}}

\newcommand{\Int}{\operatorname{Int}}

\newcommand{\Area}{\operatorname{Area}}
\newcommand{\Ann}{\opna{Ann}}

\newcommand{\SL}{_\Lambda}

\newcommand{\DL}{D\SL}

\renewcommand{\mod}{\opna{mod}}

\newcommand{\diam}{\operatorname{diam}}

\newcommand{\PR}{{\opna{PR}\SL}}
\newcommand{\CB}{\operatorname{CB}\SL}

\newcommand{\CC}{\operatorname{CC}\SL}

\newcommand{\cn}{\operatorname{cn}\SL}
\newcommand{\cm}{\operatorname{cm}\SL}
\newcommand{\ncc}{\operatorname{ncc}\SL}
\newcommand{\NC}{\operatorname{NC}}
\newcommand{\CL}{\operatorname{C\Lambda}\SL}
\newcommand{\AnnL}{\Ann\SL}

\newcommand{\thor}{{\widetilde{\opna{h}}}{\opna{or}}}
\newcommand{\tHor}{{\widetilde{\opna{H}}}{\opna{or}}}
\newcommand{\tver}{{\widetilde{\opna{v}}}{\opna{er}}}
\newcommand{\tVer}{{\widetilde{\opna{V}}}{\opna{er}}}
\newcommand{\hor}{\opna{hor}}
\newcommand{\Hor}{\opna{Hor}}
\newcommand{\ver}{\opna{ver}}
\newcommand{\Ver}{\opna{Ver}}

\newcommand{\veps}{\varepsilon}

\newcommand{\vphi}{\varphi}

\newcommand{\cV}{{\mathcal{V}}}

\newcommand{\cP}{{\mathcal{P}}}

\newcommand{\R}{\mathbb{R}}

\newcommand{\N}{\mathbb{N}}
\newcommand{\C}{\mathbb{C}}

\newcommand{\pichere}[2]
{\begin{center}\includegraphics[width=#1\textwidth]{#2}\end{center}}

\newcommand{\lab}[3]{\psfrag{#1}[#3]{$\scriptstyle{#2}$}}

\author{Andr\'e de Carvalho} \address{Departamento de Matem\'atica
  Aplicada, IME - USP \\ Rua do Mat\~ao 1010, Cidade Universit\'aria
  \\ 05508-090 S\~ao Paulo, SP, Brazil}\email {andre@ime.usp.br}
\author{Toby Hall} \address{Department of Mathematical
  Sciences\\University of Liverpool\\Liverpool L69 7ZL, UK} \email
       {T.Hall@liv.ac.uk}

\title{Riemann surfaces out of paper}

\begin{document}
\maketitle

\begin{abstract}
Let~$S$ be a surface obtained from a plane polygon by identifying
infinitely many pairs of segments along its boundary. A condition is
given under which the complex structure in the interior of the polygon
extends uniquely across the quotient of its boundary to make~$S$ into
a closed Riemann surface. When this condition holds, a modulus of
continuity is obtained for a uniformizing map on~$S$.
\end{abstract}

\bibliographystyle{amsplain}

\section{Introduction}
\label{sec:introduction}

In the usual approach to their topological classification, surfaces
are constructed by making side identifications on plane polygons. Such
constructions produce surfaces with more than just a topological
structure. First, they are naturally metric spaces, with metrics that
are Riemannian and flat away from finitely many points: at the
exceptional points the metric fails to be Riemannian because there is
too much or too little angle. Thus, for example, when a genus two
surface is constructed from a regular octagon, all of the vertices of
the polygon are identified and the corresponding point on the surface
is a {\em cone} point with angle $6\pi$. Second, considering this
argument from the conformal rather than the metric standpoint, the
surfaces constructed in this way have a natural conformal structure:
at flat points this comes from the Euclidean structure, while at cone
points an appropriate fractional power can be used to introduce local
coordinates.

This idea can be generalized by constructing surfaces from plane
polygons with {\em infinitely} many pairs of segments along the
boundary identified. It is natural to ask whether the above discussion
carries through to yield a conformal structure on the quotient. The
fact that infinitely many segment pairs are identified allows the
possibility of {\em singular points} (such as accumulations of cone
points) in the quotient across which it is no longer evident that the
conformal structure extends. In addition, the set of singular points
may itself have some topological complexity -- for instance, it may be
a Cantor set. These questions are not only of interest in their own
right in the theory of Riemann surfaces, but have important
applications in dynamical systems theory, which was the original
motivation for this research. As explained below, surfaces constructed
in this manner play an important role in the study of the dynamics of
surface automorphisms, and the above questions are crucial for the
construction of limits in families of such automorphisms.

The first main theorem of the paper provides a sufficient condition
for the conformal structure to extend across the singular set, making
the quotient into a closed Riemann surface. This condition is that a
certain integral diverges at every point of the singular set. The
second main theorem uses the rate of divergence of the integral to
obtain a modulus of continuity for suitably normalized uniformizing
maps on the Riemann surfaces.

\medskip
To be more precise, let $P$ be a finite collection of disjoint
polygons in the complex plane. A {\em paper-folding scheme} is an
equivalence relation which glues together segments, possibly
infinitely many, along the boundary of $P$: the union of the segments
is required to have full measure. The image of $\partial P$ in the
quotient space $S$ is called the {\em scar\/}: it may contain {\em cone
  points}, where the total angle is not equal to~$2\pi$, and {\em
  singular points}, such as accumulations of cone points. Necessary
and sufficient conditions can be given for the quotient space to be a
surface, and the following statements summarize the main theorems of
this paper in the surface case.

\bigskip

\noindent\textbf{Conformal Structure Theorem
  (Theorem~\ref{thm:generalmain})} \,\, \emph{The natural conformal
  structure on the conic-flat part of $S$ extends uniquely to all
  of~$S$ provided that a certain integral diverges at each point of
  the singular set. }
\bigskip

It should be emphasized that this theorem only provides a sufficient
condition for the extension of the conformal structure across the
singular set. The authors do not know any specific examples of
surfaces formed in this way (with the singular set having zero
push-forward measure) for which the natural conformal structure does
not extend.

If the conditions of Theorem~\ref{thm:generalmain} are satisfied,
then~$S$ has the structure of a closed Riemann surface. In the case
where the paper-folding scheme is {\em plain}, $S$ is topologically a
sphere, and an explicit modulus of continuity of a suitably normalized
uniformizing map from~$S$ to the Riemann sphere is found. The modulus
of continuity is obtained by patching together local moduli of
continuity about each point of the surface -- these local moduli carry
over to Riemann paper surfaces other than spheres, and could be
used to construct global moduli in these cases also.

\bigskip

\noindent\textbf{Modulus of continuity of uniformizing map
  (Theorem~\ref{thm:mod-cont})} \,\, \emph{Suitably normalized
  uniformizing maps on Riemann paper spheres have moduli of continuity
  which depend only on the geometry of the polygons~$P$ and the metric
  on the scar.}

\bigskip

These results generalize and complete those set forth in~\cite{O1},
where only finite singular sets were considered.  There the
topological and metric structures on the quotients of paper foldings
were studied and theorems were proved which guarantee that, under
appropriate conditions, the quotient space is indeed a closed surface.
In~\cite{O1} the singular points, being isolated, could be considered
independently; the extension of the conformal structure and the local
modulus of continuity at a singular point were obtained from an
analysis of the geometry of a system of nested annuli zooming down to
the point in question. This approach is no longer viable in the case
of general singular sets, and is replaced in this paper with the
construction and analysis of more complicated systems of annuli: the
main tool used is a criterion of McMullen~\cite{Curt} for a set to
have absolute area zero (Theorem~\ref{thm:absareazero} below).

Thurston's pseudo-Anosov maps~\cite{Thu} provide a natural class of
examples of surface automorphisms which are defined on (finite) paper
surfaces: since every pseudo-Anosov admits a Markov partition, its
surface of definition can be constructed from the disjoint union of
Markov rectangles by identifying segments along their boundaries. The
results of this paper give a means to construct limits of sequences
of pseudo-Anosovs defined on the same underlying topological surface,
but with puncture sets of varying cardinality. 

In one case which motivated this work, the pseudo-Anosovs are the
canonical Thurston representatives of the isotopy classes of Smale's
horseshoe map~\cite{Smale} of the sphere relative to certain periodic
orbits. Provided that the conditions of Theorem~\ref{thm:mod-cont}
apply uniformly along the sequence, the suitably normalized
uniformizing maps from the abstract spheres of definition of the
pseudo-Anosovs to the Riemann sphere have a uniform modulus of
continuity. This enables Arzel\`a-Ascoli type arguments to be used to
constuct limiting automorphisms of the Riemann sphere. In such
examples the number of identifications along the boundaries of the
polygons is finite, although unbounded along the sequence of
pseudo-Anosovs: similar techniques can also be applied to the {\em
  generalized pseudo-Anosovs} of~\cite{ugpa}, for which an infinite
number of identifications is required.

\medskip
Section~\ref{sec:background} contains a brief summary of necessary
definitions and results from metric geometry and geometric function
theory. The definition of paper surfaces, together with some of their
key properties, is given in Section~\ref{sec:paper-surfaces}.

In Section~\ref{sec:conf-struct-paper}, a sufficient condition for the
conformal structure on the set of non-singular points of a paper
surface to extend over an arbitrary singular set is given
(Theorem~\ref{thm:generalmain}).

In Section~\ref{sec:modulus-continuity}, attention is restricted to
Riemann paper spheres, and the modulus of continuity of a suitably
normalized isomorphism to the Riemann sphere is studied. An explicit
modulus of continuity is given at each point of the paper sphere
(Theorem~\ref{thm:local-modulus-continuity}), and it is shown that
these give rise to a global modulus of continuity
(Theorem~\ref{thm:mod-cont}).

\bigskip\noindent {\bf Acknowledgements:} The authors would like to
thank Lasse Rempe for his very helpful comments on a draft of this
paper. They gratefully acknowledge the support of FAPESP
grant 2006/03829-2. The first author would also like to
acknowledge the support of CNPq 
grant 151449/2008-2 and the hospitality
of IMPA,
where part of this work was developed.

\section{Background}
\label{sec:background}

In this section some necessary background from metric geometry and
geometric function theory is briefly summarised. The book of Burago,
Burago, and Ivanov~\cite{BBI} is an excellent reference for the
former; and the texts of Ahlfors~\cite{Ah}, Ahlfors-Sario~\cite{AhSa},
and Lehto-Virtanen~\cite{LeVi} are the classic references for the
latter.

\subsection{Metric spaces: quotient metrics and intrinsic metrics}
\label{sec:metr-spac-quot}

Let~$(X,d)$ be a metric space: it is convenient to take $d\co X\times
X\to\R_{\ge 0}\cup\{\infty\}$, so that two points can infinitely
distant from each other. If $x\in X$ and~$\Lambda$ is a compact
subset of~$X$, then the notation~$d(x,\Lambda)$ is used to denote the
distance from~$x$ to the closest point of~$\Lambda$. The open and
closed balls of radius~$r$ about~$\Lambda$ are denoted
\begin{eqnarray*}
B_X(\Lambda;r) &=& \{y\in X\,:\, d(y,\Lambda)<r\} \qquad\text{and}\\
\ol{B}_X(\Lambda;r) &=& \{y\in X\,:\, d(y,\Lambda)\le r\}.
\end{eqnarray*}

Let~$\rmR$ be a relation on~$X$ (which in this paper will always be
reflexive and symmetric). The {\em
  quotient metric space of~$X$ under~$\rmR$} is the
quotient~$(X/d^{\rmR}, d^{\rmR})$ of~$X$ by the greatest
semi-metric~$d^{\rmR}$ on~$X$ for which $d^{\rmR}(x,y)=0$ whenever
$x\rmR y$, and $d^{\rmR}(x,y)\le d(x,y)$ for all~$x,y\in X$
(so~$d^{\rmR}$ is the supremum of all semi-metrics which satisfy these
conditions).

The length of a path $\gamma\co[a,b]\to X$ is defined to be
\[|\gamma|_X = \sup \left\{
\sum_{i=1}^{k-1} d(\gamma(t_i), \gamma(t_{i+1}))
\right\} \in\R_{\ge 0}\cup\{\infty\},
\]
where the supremum is over all finite increasing
sequences~$t_1\le t_2\le \cdots\le t_k$
in~$[a,b]$. If~$|\gamma|_X<\infty$ then~$\gamma$ is said to be {\em
  rectifiable}. 

The metric~$d$ on~$X$ is said to be {\em intrinsic} if the distance
between any two points is arbitrarily well approximated by the lengths
of paths joining them; and to be {\em strictly intrinsic} if the
distance between any two points is equal to the length of a path
joining them. It can be shown that a compact intrinsic metric is
strictly intrinsic.

If~$d$ is not an intrinsic metric, then there is an induced intrinsic
metric $\hd$ on~$X$ defined by
\[\hd(x,y) = \inf\{
|\gamma|_X\,:\,\gamma \text{ is a path from~$x$ to~$y$}
\}
\]
(and $\hd(x,y)=\infty$ if $x$ and $y$ lie in different path components
of~$X$). 
For example, if~$P\sbs\C$ is a polygon, then there is an intrinsic
metric~$d_P$ on~$P$ induced by the Euclidean metric on~$\C$, which
does not in general agree with the subspace metric on~$P$.

A proof of the following result can be found on pp. 62\,--\,63
of~\cite{BBI}.

\begin{lem}
\label{lem:intrinsic-quotient is intrinsic}
Let~$(X,d)$ be an intrinsic metric space, and~$\rmR$ be a relation
on~$X$. Then the quotient metric space~$(X/d^{\rmR}, d^{\rmR})$ is
also intrinsic.
\end{lem}

\subsection{Background geometric function theory}

\begin{defns}[Module of an annular region, concentric and nested annular regions]
\label{defns:annulus-stuff}
An {\em annular region} in a surface is a subset homeomorphic to an
open round annulus
\[A(r_1,r_2) = \{z\in\C\,:\, r_1<|z|<r_2\},\]
where $0\le r_1<r_2\le\infty$. If~$A\sbs\C$ is an annular region,
then there is a conformal map taking~$A$ onto some round
annulus~$A(r_1,r_2)$, and this map is unique up to postcomposition
with a homothety. The ratio $r_2/r_1$ is therefore a conformal
invariant of~$A$, and the {\em module} $\mod A$ of~$A$ is defined by
\[
\mod A = 
\left\{
\begin{array}{ll}
\infty & \text{ if }r_1=0 \text{ or }r_2=\infty,\\
\frac{1}{2\pi} \ln\frac{r_2}{r_1} \quad & \text{ otherwise.}
\end{array}
\right.
\]

Suppose that~$D$ is a closed topological disk and that~$A\sbs D$ is an
annular region. The {\em bounded component} of~$D\setminus A$ is the
one which is disjoint from~$\partial D$.  

Annular regions $A_0,A_1\sbs D$ are {\em concentric} if the bounded
complementary component of one is contained in the bounded
complementary component of the other. They are {\em nested} if one is
entirely contained in the bounded complementary component of the
other.
\end{defns}

The following fundamental inequality will be important:
\begin{thm}[Additivity of the module]
\label{thm:additivity-module}
If~$A_n$ is a finite or countable family of nested annular regions all
contained in and concentric with the annular region~$A$, then
\[\mod A \ge \sum \mod A_n.\]
\end{thm}

A {\em conformal metric} on $\C$ is a metric obtained defining the
length of arcs by $|\gamma|_\nu:=\int_\gamma\nu(z)|\!\rmd z|$, where
$\nu$ is a nonnegative Borel measurable function defined on $\C$. If
$A$ is an annular region in~$\C$ whose boundary components are
$C_1,C_2$ then
\[\mod A= \sup\left\lbrace\dfrac{d_\nu(C_1,C_2)^2}{\Area_\nu(A)}\,:\, 
\nu(z)|\!\rmd z| \text{ is a conformal metric}\right\rbrace,\] where
$d_\nu(C_1,C_2)$ is the $\nu$-distance between the boundary components
of $A$ (i.e., the minimum $\nu$-length of an arc with endpoints
in~$C_1$ and~$C_2$) and $\Area_\nu(A)=\iint_A\nu(z)^2\rmd x\rmd y$ is
the $\nu$-area of~$A$.  In particular, given any conformal metric
$\nu(z)|\!\rmd z|$, 
\begin{equation}\label{eq:modest}
\mod A\geq \dfrac{d_\nu(C_1,C_2)^2}{\Area_\nu(A)}.
\end{equation}

The following theorem provides an upper bound on the module of an
annulus. 

\begin{thm}[Gr\"otzsch Annulus Theorem]
\label{thm:GAT}
Let~$A$ be an annular region contained in the disk~$|z|\le R$ in~$\C$,
and suppose that the points~$z_1$ and~$z_2$ are contained in the
bounded complementary component of~$A$. Then
\[\mod A \le \frac{1}{2\pi} \ln\frac{8R}{|z_1-z_2|}.\]
\end{thm}
\qed

A proof of the following straightforward distortion theorem can be
found in~\cite{O1}. 

\begin{thm}\label{thm:Koebedistn}
Let $P\sbs\C$ be a closed topological disk, $p\in\Int(P)$ and
$\Phi\co \Int(P)\to\C$ be a univalent (i.e.\ injective and
holomorphic) function with $\Phi(p)=0$ and $\Phi'(p)=1$. Let
$Q(h)$ denote the closed interior collar neighborhood of $\partial
P$ of $d_\C$-width~$h>0$ and set $P_h:=P\setminus Q(h)$. Assume that
$h$ is small enough that $P_h$ is path connected and $p\in P_h$, and
set
\[\kappa:=\exp\left(\frac{8\diam_{P_h}P_h}{h}\right),\]
where $\diam_{P_h}P_h$ denotes the diameter of $P_h$ in the intrinsic
metric $d_{P_h}$ on $P_h$ induced by~$d_\C$.
Then $\Phi$ is $\kappa$-biLipschitz in $P_h$.
\end{thm}
\qed

A compact subset $E\sbs\C$ whose complement $W_E:=\C\setminus E$ is
connected has {\em absolute area zero} if, for any univalent
map $f\co W_E\to\C$, the area of $\C\setminus f(W_E)$ is
zero. In terms of the classification of Riemann surfaces, this is
equivalent to saying that $W_E\in O_{AD}$, the class of Riemann
surfaces whose only analytic functions with bounded Dirichlet integral
are the constants (Theorem IV.2B, p.\ 199 of~\cite{AhSa}). This
implies, in particular, the following {\em removability} theorem (see
Theorem IV.4B, p.\ 201 of~\cite{AhSa}):

\begin{thm}\label{thm:removability1} 
Let $\Omega\sbs\C$ be a domain and $E\sbs\Omega$ be a subset with
absolute area zero such that $\Omega\setminus E$ is also a
domain. Then any bounded univalent map $g\co\Omega\setminus E\to\C$
extends uniquely to all of~$\Omega$ as a bounded univalent map.
\end{thm}
\qed

The following criterion for a set to have absolute area zero
is due to McMullen~\cite{Curt}:
\begin{thm}\label{thm:absareazero}
Let $U_1,U_2,\ldots$ be a sequence of disjoint open sets in $\C$
satisfying the following conditions:
\begin{enumerate}[a)]
\item $U_n$ is a finite union of disjoint unnested annular regions of
  finite module;
\item any component of $U_{n+1}$ is nested inside some component of
  $U_n$;
\item if $\{A_n\}$ is any nested sequence of annular regions,
  where $A_n$ is a component of $U_n$, then $\sum_n\mod(A_n)=\infty$.
\end{enumerate}
Let $E_n$ denote the union of the bounded components of $\C\setminus
U_n$ and set $E:=\bigcap_n E_n$. Then $E$ is a totally disconnected
compact set of absolute area zero. 
\end{thm}
\qed

\section{Paper surfaces}
\label{sec:paper-surfaces}
In this section the main objects of study -- paper-folding schemes and
paper surfaces -- are introduced. Their topological and metric
structures are discussed, and some pertinent results from~\cite{O1}
are stated without proof: the emphasis in this paper will be on the
conformal structure of paper surfaces, which is discussed in
Section~\ref{sec:conf-struct-paper}. 

\begin{defns}[Arc, segment, polygon, multipolygon]
An {\em arc} in a metric space~$X$ is a homeomorphic image $\gamma\sbs
X$ of the interval $[0,1]$.  Its {\em endpoints} are the images of $0$
and $1$ and its {\em interior}~$\interior{\gamma}$ is the image of
$(0,1)$. A {\em segment} is an arc in $\C$ which is a subset of a
straight line. The length of a segment $\alpha$ is denoted $|\alpha|$.
A {\em simple closed curve} in $X$ is a homeomorphic image of the unit
circle.

An arc or simple closed curve in~$\C$ is called {\em polygonal} if
it is the concatenation of finitely many segments. The maximal
segments are called the {\em edges} of the arc or simple closed curve,
and their endpoints are its {\em vertices}.

 A {\em polygon} is a closed topological disk in $\C$ whose boundary
 is a polygonal simple closed curve. Its {\em vertices} are the same
 as its boundary's vertices and its {\em sides} are the edges forming
 its boundary.

A {\em multipolygon} is a disjoint union of finitely many polygons.  A
{\em polygonal multicurve} is a disjoint union of finitely many
polygonal simple closed curves. 
\end{defns}

\begin{defns}[Segment pairing, interior pair, full collection, fold]
Let $C\sbs\C$ be an oriented polygonal multicurve and
$\alpha,\alpha'\sbs C$ be segments (not necessarily edges) of the same
length with disjoint interiors. The {\em segment pairing}
$\left\langle\alpha,\alpha'\right\rangle$ is the relation which
identifies pairs of points of $\alpha$ and $\alpha'$ in a
length-preserving and orientation-reversing way. The segments
$\alpha,\alpha'$ and any two points which are identified under the
pairing are said to be {\em paired}. Two paired points which lie in
the interior of a segment pairing form an {\em interior pair}. The
notation for a pairing is not ordered, so that
$\left\langle\alpha,\alpha'\right\rangle$ and
$\left\langle\alpha',\alpha\right\rangle$ represent the same pairing.

A collection $\{\left\langle\alpha_i,\alpha_i'\right\rangle\}$ of
segment pairings is {\em interior disjoint} if the interiors of all of
the segments $\alpha_i$ and $\alpha_i'$ are pairwise disjoint.

The {\em length} of a segment pairing
$\llangle\alpha,\alpha'\rrangle$, denoted
$|\llangle\alpha,\alpha'\rrangle|$, is the length of one of the arcs
in the pairing, i.e.,
$|\llangle\alpha,\alpha'\rrangle|:=|\alpha|=|\alpha'|$.

A (necessarily countable) interior disjoint collection
$\cP=\{\llangle\alpha_i,\alpha'_i\rrangle\}$ of segment pairings
on~$C$ is {\em full} if $\sum_i|\llangle\alpha_i,\alpha'_i\rrangle|$
equals half the length of~$C$, so that the pairings in $\cP$ cover $C$
up to a set of Lebesgue 1-dimensional measure zero.

A pairing of two segments which have an endpoint in common is called a
{\em fold} and the common endpoint is called its {\em folding
  point}. The folding point of a fold is therefore not paired with
any other point.

An interior disjoint collection~$\cP$ of segment pairings induces a
reflexive and symmetric {\em pairing relation}, also denoted~$\cP$, so
$\cP = \{(x,x'):\text{ $x,x'$ are paired or $ x=x'$}\}$.
\end{defns}

\begin{defns}[Paper-folding scheme]
\label{defn:origami}
A {\em paper-folding scheme} is a pair $(P,\cP)$ where $P\sbs\C$ is
a multipolygon with the intrinsic metric $d_P$ induced from $\C$,
and $\cP$ is a full interior disjoint collection of segment pairings
on $\partial P$ (positively oriented). The metric quotient
$S:=P/d_P^\cP$ of $P$ under the semi-metric $d_P^\cP$ induced by the
pairing relation $\cP$ is the associated {\em paper space}. When $S$
is a closed (compact without boundary) topological surface, then
$(P,\cP)$ is called a {\em surface paper-folding scheme} and $S$ is
the associated {\em paper surface}.

The projection map is denoted $\pi\co P\to S$ and the quotient
$G=\pi(\partial P)\sbs S$ of the boundary is the {\em scar}. Notice
that the restriction $\pi:\Int(P)\to S\setminus G$ is a homeomorphism.

The (quotient) metric on~$S$ is denoted~$d_S$. The metric~$d_G$ on~$G$
is defined to be its intrinsic metric as a subset of~$S$: this is
equal to the quotient metric on~$G$ as a quotient of~$\partial P$,
where $\partial P$ is endowed with its intrinsic metric as a subset
of~$P$ (Lemma~\ref{lem:metricstructureofG}).

The measure $\rmm_G$ on~$G$ is defined to be the push-forward of
Lebesgue $1$-dimensional measure $\rmm_{\partial P}$ on $\partial
P$. Hausdorff $1$-dimensional measure on~$G$ is denoted~$\mu^1_G$ ---
Lemma~\ref{lem:metricstructureofG} states that $\mu^1_G=
\frac{1}{2}\rmm_G$.
\end{defns}

\medskip

 The next definitions distinguish different types of points in a paper
 space.

\begin{defns}[Vertex, edge, singular point, planar point]
\label{defn:origamipts}
For $k\in\N\cup\{\infty\}$, a point $x\in G$ is a {\em vertex of
  valence~$k$}, or a {\em $k$-vertex}, if either (i) $\#\pi\I(x)=k\neq
2$; or (ii) $\#\pi\I(x)=k=2$ and $\pi\I(x)$ contains a vertex of
$P$. Let~$\cV$ denote the set of all vertices of~$G$.

The points of the paper space~$S$ are divided into three types:
\begin{description}
\item[Singular points] vertices of valence~$\infty$ and accumulations
  of vertices. Let~$\cV^s$ denote the set of singular points.
\item[Regular vertices] vertices which are not singular (that is,
  isolated vertices of finite valence).
\item[Planar points] all other points of~$S$: that is, the points of
  $S\setminus \bcV$.
\end{description}

The closures of the connected components of $G\setminus\bcV$ are
called {\em edges} of the scar~$G$.
\end{defns}

The next lemma summarises the properties of the metric and measure on
the scar~$G$ which will be used here.

\begin{lem}[\cite{O1}]
\label{lem:metricstructureofG}
Let~$G$ be the scar of the paper-folding scheme~$(P,\cP)$. Then
\begin{enumerate}[a)]
\item The set of planar points is open and dense in the scar~$G$, while the
set~$\bcV$ of vertices and singular points is a closed nowhere dense subset
of~$G$ with zero $\rmm_G$-measure.
\item The intrinsic metric on~$G$ induced by the inclusion $G\subset
  S$ agrees with the quotient metric on $G=\partial P / d_{\partial
    P}^{\cP}$ induced by the intrinsic metric~$d_{\partial P}$
  on~$\partial P$.
\item $G$ has Hausdorff dimension 1, and Hausdorff $1$-dimensional
  measure $\mu_G^1$ on~$G$ is equal to~$\frac{1}{2}\rmm_G$.
\item Every arc~$\gamma$ in~$G$ is rectifiable, and $|\gamma|_G =
  \frac{1}{2}\rmm_G(\gamma)$. 

\end{enumerate}
\end{lem}
\qed

 For clarity of exposition, attention will be concentrated on {\em
   plain} folding schemes: these are both the most common and the
 simplest type of paper foldings and for them the paper space is
 always a sphere and the scar is always a dendrite
 (Theorem~\ref{thm:plain-top-structure}).

\begin{defns}[Unlinked pairing, plain paper-folding scheme]
\label{defns:linked-plain}
Let~$\gamma$ be a polygonal arc or polygonal simple closed curve.

Two pairs of (not necessarily distinct) points $\{x,x'\}$ and
$\{y,y'\}$ of~$\gamma$ are {\em unlinked} if one pair is contained in
the closure of a connected component of the complement of the
other. Otherwise they are {\em linked}.

A symmetric and reflexive relation~$\rmR$ on $\gamma$ is {\em
  unlinked} if any two unrelated pairs of related points are unlinked:
that is, if $x\,\rmR\,x'$, $y\,\rmR\,y'$, and neither $x$ nor~$x'$ is
related to either~$y$ or~$y'$, then $\{x,x'\}$ and $\{y,y'\}$ are
unlinked. 

An interior disjoint collection~$\cP$ of segment pairings on~$\gamma$
is {\em unlinked} if the corresponding relation~$\cP$ is unlinked.

A paper-folding scheme~$(P,\cP)$ is {\em plain} if~$P$ is a single
polygon and~$\cP$ is unlinked.
\end{defns}

\begin{thm}[Topological structure of a plain paper folding~\cite{O1}]
\label{thm:plain-top-structure}
The quotient $S$ of a plain paper-folding scheme is a topological
sphere, and its scar~$G$ is a dendrite.
\end{thm}
\qed

\begin{rmks}
\label{rmk:dendrites}
Recall that a {\em dendrite}~$G$ is a locally connected continuum
which doesn't contain any simple closed curve. The following
properties of dendrites will be used later~\cite{KuraII,Whyburn}:
\begin{enumerate}[a)]
\item Any two distinct points of~$G$ are separated by a third point
  of~$G$. Conversely, any continuum with this property is a dendrite.
\item If~$x$ and~$y$ are distinct points of~$G$ then there is a unique
  arc in~$G$, denoted~$[x,y]_G$, with endpoints~$x$ and~$y$. The
  notation~$(x,y)_G$ will be used to denote~$[x,y]_G\setminus\{x,y\}$.
\item Every subcontinuum of~$G$ is also a dendrite.
\end{enumerate}
\end{rmks}

Necessary and sufficient conditions can also be given for a
paper-folding scheme to be a surface paper folding. The only
constraint is the obvious one: linked pairings create handles, and if
there are infinitely many handles then the quotient space cannot be a
compact surface -- the paper-folding scheme must therefore be
``finitely linked''. However, all that will be needed here is the
following description of the scar of a surface paper folding.

\begin{thm}[\cite{O1}]
\label{thm:top-structure}
Necessary and sufficient conditions can be given for a paper-folding
scheme~$(P,\cP)$ to be a surface paper folding
scheme. In this case, the
scar~$G$ is a local dendrite, which can be written as $G=C\cup\Gamma$,
where
\begin{itemize}
\item $C$ is a finite connected graph in~$S$ with the property that
  any simple closed curve in~$C$ is homotopically non-trivial in~$S$; and
\item $\Gamma$ is a union of finitely or countably many disjoint
  dendrites, with diameters decreasing to~$0$, each of which
  intersects~$C$ in exactly one point.
\end{itemize}
\end{thm}
\qed

\section{Conformal structures on paper surfaces}
\label{sec:conf-struct-paper}

If~$(P,\cP)$ is a surface paper-folding scheme with quotient paper
surface~$S$, then there is a natural complex structure on the set
$S\setminus\ol\cV$ of planar points of~$S$ induced by the local
Euclidean structure. This complex structure extends readily across
regular (isolated) vertices of~$G$: at such a vertex around which the
total angle is~$\theta$, maps of the form $z\mapsto z^{2\pi/\theta}$
can be used as charts. 

The question addressed in this section is whether this conformal
structure on $S\setminus\cV^s$ extends uniquely across~$\cV^s$ to
endow~$S$ with a unique natural Riemann surface structure. When this
is the case, $S$ is said to be a {\em Riemann paper surface}. 

In fact, it can be useful to ask whether the conformal structure
extends across some subset~$\Lambda$ of the singular set. Throughout
this section, $(P,\cP)$ will be a surface paper-folding scheme with
scar~$G$, and~$\Lambda$ will denote a clopen (i.e.\ open and closed)
subset of the set~$\cV^s\sbs\ol\cV$ of singular points of~$G$. By
Lemma~\ref{lem:metricstructureofG}, any such set~$\Lambda$ is totally
disconnected with $\meas{\Lambda}=0$.

\subsection{Statement of results}
Theorem~\ref{thm:generalmain} below is the first main result of this
paper: it gives sufficient conditions for the complex structure
on~$S\setminus\cV^s$ to extend uniquely across the set~$\Lambda$. The
theorem will be proved in Section~\ref{subsec:proof-generalmain} after
some examples have been discussed in Section~\ref{subsec:examples}.

\begin{defns}[$\CBqr$, $\CCqr$, $\cmqr$, $\cnqr$]
\label{defn:G-circles-etc}
Let~$q\in\Lambda$. Denote by $\CBqr$ the connected component of
$B_G\lr$ which contains~$q$; by $\CCqr$ the boundary in~$G$ of
$\CBqr$; by $\cmqr=\meas{\CBqr}$ the measure of $\CBqr$; and
by~$\cnqr\in\N\cup\{\infty\}$ the cardinality of~$\CCqr$.
\end{defns}

\begin{rmk}
\label{rmk:chains-in-Lambda}
If $q_1,q_2,\ldots,q_n\in\Lambda$ are such that
  $d_G(q_i,q_{i+1})<2r$ for $1\le i<n$, then all of the sets
  $\CB(q_i;r)$ are equal.
\end{rmk}

\begin{thm}
\label{thm:generalmain}
Let~$(P,\cP)$ be a surface paper-folding scheme with associated paper
surface~$S$ and scar~$G$, and let~$\Lambda\sbs\cV^s$ be a
clopen subset of~$\cV^s$. Then the complex structure
on~$S\setminus\cV^s$ extends uniquely across~$\Lambda$ provided that,
for every~$q\in\Lambda$,
\begin{equation}
\label{eq:generaldivint1}
\int_0 \frac{\rmd r}
{\cmqr + r\cdot \cnqr} = \infty.
\end{equation}
In particular, if~$\Lambda=\cV^s$ and~(\ref{eq:generaldivint1}) holds
for every~$q\in\Lambda$, then~$S$ is a Riemann paper surface.
\end{thm}

\begin{rmk}
\label{rmk:single-point}
If~$\Lambda=\{q\}$ is a single point, then $\cmqr=\meas{B_G\qr}$ and
$\cnqr$ is the number of points of~$G$ at distance~$r$ from~$q$: the
theorem thus includes the corresponding result of~\cite{O1} as a
special case.
\end{rmk}

\subsection{Examples}
\label{subsec:examples}

\begin{example}
\label{ex:simple}
In the plain paper-folding scheme of this example, there is a
countable set of singularities accumulating on a single
point. Theorem~\ref{thm:generalmain} will be used to show that the
conformal structure on the set of planar points extends uniquely to
the whole of the paper sphere~$S$.

Let~$P$ be the unit square $\{x+iy\co x,y\in[0,1]\}$ in~$\C$. The
two vertical sides of the square are paired together, and the top side
is folded in half. The segment pairings on the bottom
side of~$P$ are depicted in Figure~\ref{fig:seq-example}: in
each of the shaded regions, the pairings are shown on the
interval on the right of the figure.

\begin{figure}[htbp]
\lab{1}{\alpha_0}{b} \lab{2}{\alpha_0'}{b}
\lab{3}{\alpha_1}{b} \lab{4}{\alpha_1'}{b}
\lab{5}{\alpha_2}{b} \lab{6}{\alpha_2'}{b}
\begin{center}
\pichere{0.8}{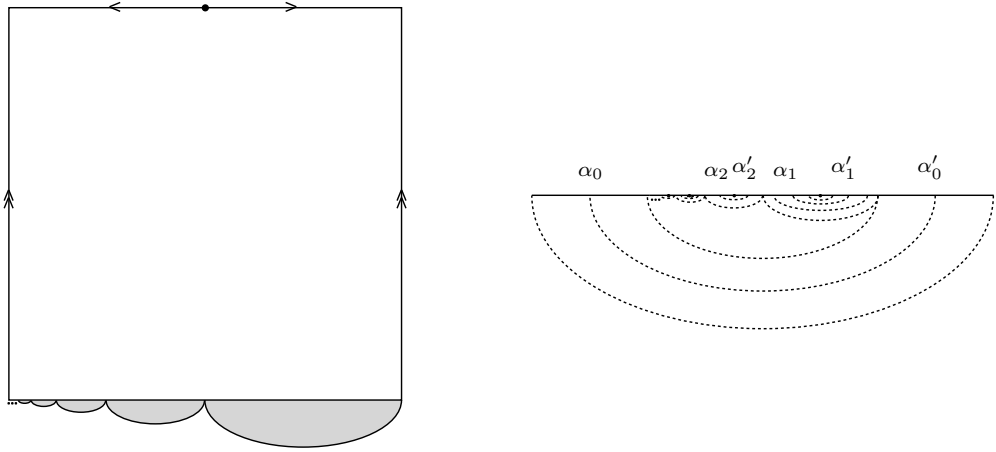}
\end{center}
\caption{A paper-folding scheme with a sequence of singularities}
\label{fig:seq-example}
\end{figure}

To describe the pairings on that interval, define
subintervals~$\alpha_k$ and $\alpha'_k$ of~$[0,1]$ for $k\ge 0$ by
$\alpha_0=[0,1/4]$, $\alpha_0'=[3/4,1]$, and 
\[\alpha_k=\left[\frac14+\frac{1}{2^{k+1}}, \frac14+\frac{3}{2^{k+2}}\right], \quad
\alpha_k'=\left[\frac14+\frac{3}{2^{k+2}}, \frac14+\frac{1}{2^k}\right] \qquad
(k\ge 1),
\]
so that $|\alpha_k|=|\alpha_k'|=1/2^{k+2}$ and $\sum_{k\ge
  0}|\alpha_k|=1/2$. For~$k\ge 1$, write
$\xi_k=\frac14+\frac{3}{2^{k+2}}$, the common endpoint of~$\alpha_k$
and~$\alpha_k'$. 

Now define embeddings $\vphi_j\co[0,1]\to\C$ for $j\ge 0$ by
$\vphi_j(x)=(x+1)/2^{j+1}$, so that $\vphi_j([0,1])=[1/2^{j+1},
  1/2^{j}]$. The segment pairings on the bottom side of~$P$ are then
given by $\ssegpair{\alpha_{j,k}}$ for $j,k\ge0$, where
$\alpha_{j,k}=\vphi_j(\alpha_k)$ and
$\alpha_{j,k}'=\vphi_j(\alpha_k')$. Observe that
$|\alpha_{j,k}|=|\alpha_{j,k}'|=1/2^{j+k+3}$, so that $\sum_{j,k\ge
  0}|\alpha_{j,k}|=1/2$.

Let~$S$ be the paper sphere corresponding to this paper-folding
scheme~$(P,\cP)$ (which is clearly plain), $\pi\co P\to S$ be the
projection, and~$G$ be the scar. Write
$a_{j,k}=\pi(\alpha_{j,k})=\pi(\alpha_{j,k}')$ for $j,k\ge 0$, and
$a$ for the projection of the vertical edges of~$P$. The
scar is depicted in Figure~\ref{fig:seq-example-scar}. Its non-planar
points are as follows:
\begin{itemize}
\item A valence 1 vertex at $y=\pi(1/2+i)$ and a valence 2 vertex at
  $z=\pi(i)=\pi(1+i)$;
\item Valence 1 vertices at $x_{j,k}=\pi(\vphi_j(\xi_k))$ for $j\ge 0$
  and $k\ge 1$;
\item Singularities at $s_j = \pi(\vphi_j(1/4))$ for $j\ge 0$; and
\item A singularity at $s = \pi(0)=\pi(1)=\pi(1/2)=\pi(1/4)=\cdots$.

\begin{figure}[htbp]
\lab{1}{1}{} \lab{2}{1/2}{} \lab{4}{1/8}{} \lab{5}{1/16}{}
\lab{6}{1/32}{}
\lab{s}{s}{} \lab{t}{s_0}{} \lab{u}{s_1}{}
\lab{a}{a_{0,0}}{} \lab{b}{a_{0,1}}{} \lab{c}{a_{1,0}}{} \lab{d}{a}{}
\lab{x}{x_{0,1}}{} \lab{y}{x_{0,2}}{} \lab{z}{x_{1,1}}{}
\lab{w}{x_{1,2}}{} 
\lab{p}{y}{} \lab{r}{z}{}
\begin{center}
\pichere{0.55}{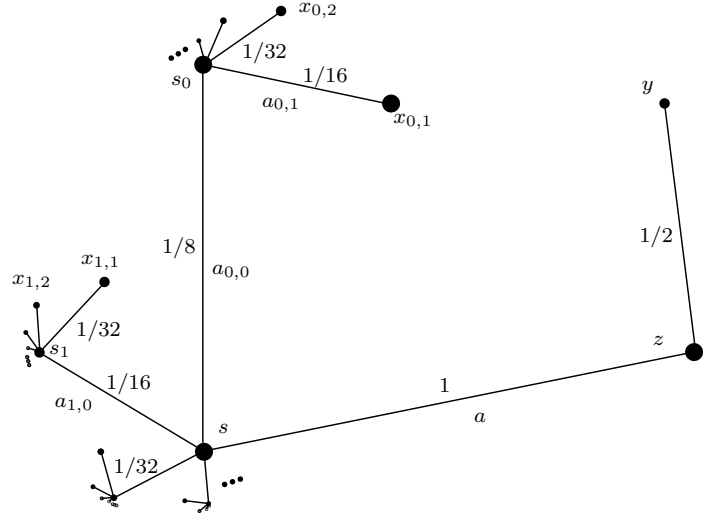}
\end{center}
\caption{The scar of the paper-folding scheme of
  Figure~\ref{fig:seq-example}}
\label{fig:seq-example-scar}
\end{figure}
\end{itemize}

Theorem~\ref{thm:generalmain} will now be applied with
$\Lambda=\cV^s=\{s\}\cup\{s_j\,:\,j\ge 0\}$. In
order to do this, it is necessary to understand the
sets~$\CBqr$ and the cardinality $\cnqr$ of
$\CCqr$ for $q\in\Lambda$.

Observe first, using Remark~\ref{rmk:chains-in-Lambda}, that if
$1/16<r\le1/8$ then~$\CBqr$ is independent
of~$q\in\Lambda$, and consists of the union of all of the
edges~$a_{j,k}$ together with a segment of the edge~$a$ of
length~$r$. Hence $\cmqr =1+2r$
and~$\cnqr=1$.

If $1/32<r\le1/16$ then there are two possibilities:
\begin{enumerate}[1)]
\item If $q=s_0$ then $\CBqr$ is the union of the edges~$a_{0,k}$ for
  $k\ge 1$, together with a segment of the edge~$a_{0,0}$ of
  length~$r$. Hence $\cmqr =1/4+2r$ and~$\cnqr=1$.
\item Otherwise, $\CBqr$ is the union of the
  edges~$a_{j,k}$ for $j\ge 1$ and $k\ge 0$, together with segments
  of~$a_{0,0}$ and~$a$ of length~$r$. Hence
  $\cmqr =1/2+4r$ and $\cnqr=2$.
\end{enumerate}
Therefore if $1/32<r\le 1/16$ then $\cmqr  +
r\cdot\cnqr \le 1/2 + 6r$.

By a similar argument, if $1/2^{j+4}<r\le 1/2^{j+3}$ for some~$j\ge 0$,
then both $\cmqr $ and $r\cdot\cnqr$ are maximised
when $s\in\CBqr$; in this case,
$\cmqr =1/2^j + 2(j+1)r$ and
$\cnqr=j+1$. Therefore
\[\cmqr +r\cdot\cnqr \le
\frac{1}{2^j}+3(j+1)r \qquad \left(\frac{1}{2^{j+4}} < r \le
\frac{1}{2^{j+3}}\right). \]

Hence, for all~$q\in\Lambda$,
\begin{eqnarray*}
\int_{1/2^{j+4}}^{1/2^{j+3}} \frac
{\rmd r}
{\cmqr +r\cdot\cnqr}
&\ge&
\int_{1/2^{j+4}}^{1/2^{j+3}} \frac
{\rmd r}
{\frac{1}{2^j}+3(j+1)r} \\
&=&
\frac{1}{3(j+1)}\ln\left(\frac{22+6j}{19+3j}\right)\\
&\ge& \frac{\ln(22/19)}{3(j+1)},
\end{eqnarray*}
so that $\displaystyle{\int_0 \frac
{\rmd r}
{\cmqr +r\cdot\cn\qr}}$ diverges.

\medskip\medskip

It follows from Theorem~\ref{thm:generalmain} that there is a unique
conformal structure on the paper sphere~$S$ which agrees with that
induced by the Euclidean structure on the set of planar points.

\end{example}

\begin{rmk}
If the sequence of lengths in the above example decreases polynomially
rather than exponentially, then the integral converges. The authors do
not know whether or not the conformal structure extends uniquely
across $\cV^s$ in this case.
\end{rmk}

\begin{example}
\label{ex:cantor}
An example is now sketched in which the paper sphere~$S$ has a Cantor
set of singularities across which the conformal structure extends
uniquely. 

As in Example~\ref{ex:simple}, the polygon~$P$ is the unit
square~$\{x+iy\,:\,x,y\in[0,1]\}$ in~$\C$, the two vertical sides
of~$P$ are paired together, and the top side is folded in half. 

To describe the pairings on the bottom side~$[0,1]\sbs\C$ of~$P$,
first set~$E_0=[2/3,1]$, and then let $E_{i,j}$ for $i\ge 1$ and $0\le
j <2^{i-1}$ be a family of disjoint closed intervals with
$|E_{i,j}|=2/3^{i+1}$ chosen as in the construction of the standard
middle-thirds Cantor set in~$[0,2/3]$: thus $E_{1,0}=[2/9,4/9]$,
$E_{2,0}=[2/27,4/27]$, $E_{2,1}=[14/27, 16/27]$, and so on: the
intervals $E_{i,j}$ are the middle thirds of the complementary
components in $[0,2/3]$ of the union of all intervals $E_{i',j'}$ with
$i'<i$. (See the top section of Figure~\ref{fig:cantor-example}.)

Now subdivide each~$E_{i,j}$ into four subintervals of equal
length~$1/(2\cdot 3^{i+1})$, denoted $\beta_{i,j}$, $\alpha_{i,j}$,
$\alpha'_{i,j}$, and $\gamma_{i,j}$ from left to right. 

The pairings along the bottom side of~$P$ can be described succintly
as follows: $\alpha_{i,j}$, $\beta_{i,j}$ and $\gamma_{i,j}$ are
paired with $\alpha'_{i,j}$, $\beta'_{i,j}$ and $\gamma'_{i,j}$, where
$\beta'_{i,j}$ and $\gamma'_{i,j}$ are subintervals of~$E_0$ chosen in
such a way that the paper-folding scheme is plain.

To be more precise, write $\beta_{i,j}=[\xi_{i,j},\zeta_{i,j}]$,
$\alpha_{i,j}=[\zeta_{i,j}, \lambda_{i,j}]$,
$\alpha'_{i,j}=[\lambda_{i,j}, \zeta'_{i,j}]$, and
$\gamma_{i,j}=[\zeta'_{i,j}, \eta_{i,j}]$. Then
$\beta'_{i,j}:=[\zeta''_{i,j},\xi'_{i,j}]=[1-\lambda_{i,j}/2,
  1-\xi_{i,j}/2]$, and $\gamma'_{i,j}:=[\eta'_{i,j},
  \zeta''_{i,j}]=[1-\eta_{i,j}/2, 1-\lambda_{i,j}/2]$. The pairings
along the bottom side of~$P$ are then given by
$\ssegpair{\alpha_{i,j}}$, $\ssegpair{\beta_{i,j}}$, and
$\ssegpair{\gamma_{i,j}}$ for $i\ge 1$ and $0\le j < 2^{i-1}$ (see
Figure~\ref{fig:cantor-example}).

\begin{figure}[htbp]
\lab{0}{0}{}
\lab{1}{1}{}
\lab{2}{2/3}{}
\lab{3}{E_0}{}
\lab{4}{E_{1,0}}{}
\lab{5}{E_{2,0}}{}
\lab{6}{E_{2,1}}{}
\lab{7}{E_{i,j}}{}
\lab{8}{E_0}{}
\lab{9}{\xi_{i,j}}{}
\lab{a}{\zeta_{i,j}}{}
\lab{b}{\lambda_{i,j}}{}
\lab{c}{\zeta'_{i,j}}{}
\lab{d}{\eta_{i,j}}{}
\lab{e}{\eta'_{i,j}}{}
\lab{f}{\zeta''_{i,j}}{}
\lab{g}{\xi'_{i,j}}{}
\lab{h}{\beta_{i,j}}{}
\lab{i}{\alpha_{i,j}}{}
\lab{j}{\alpha'_{i,j}}{}
\lab{k}{\gamma_{i,j}}{}
\lab{l}{\gamma'_{i,j}}{}
\lab{m}{\beta'_{i,j}}{}

\begin{center}
\pichere{0.8}{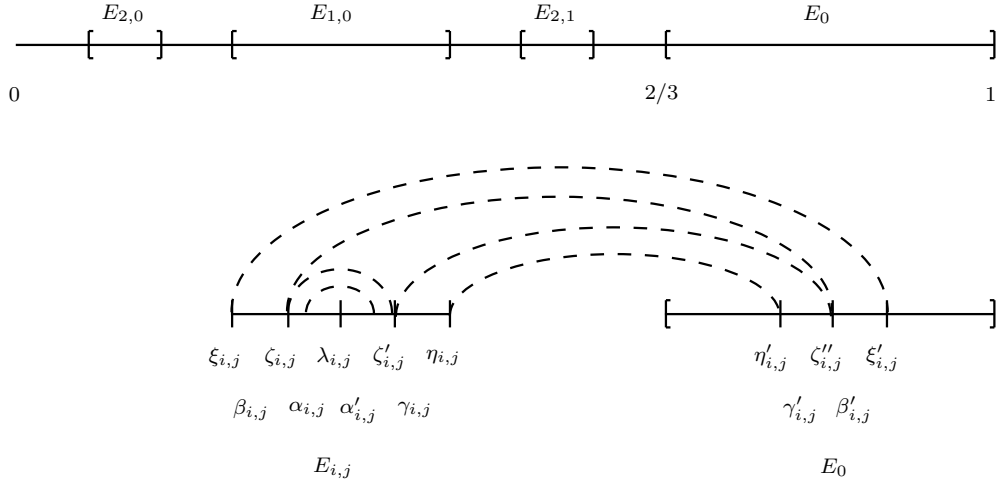}
\end{center}
\caption{Construction of a paper sphere with a Cantor set of singularities}
\label{fig:cantor-example}
\end{figure}

Let~$S$ be the paper sphere corresponding to this (plain)
paper-folding scheme, $\pi\co P\to S$ be the projection and $G$ be the
scar.  The vertices of~$G$ are all regular points, and are as follows:
\begin{itemize}
\item A valence 1 vertex at $\pi(1/2+i)$ and valence two vertices at
  $\pi(i)=\pi(1+i)$ and $\pi(0)=\pi(1)$;
\item Valence 1 vertices at $w_{i,j}:=\pi(\lambda_{i,j})$ for $i\ge 1$
  and $0\le j<2^{i-1}$; and 
\item Valence 3 vertices at
  $z_{i,j}:=\pi(\zeta_{i,j})=\pi(\zeta'_{i,j})=\pi(\zeta''_{i,j})$ for
  $i\ge 1$ and $0\le j<2^{i-1}$.
\end{itemize}
For each~$i\ge 1$ and $0\le j<2^{i-1}$, $\pi(E_{i,j})$ is a triod with
central vertex~$z_{i,j}$, and ends $w_{i,j}$,
$x_{i,j}:=\pi(\xi_{i,j})=\pi(\xi'_{i,j})$, and
$y_{i,j}:=\pi(\eta_{i,j})=\pi(\eta'_{i,j})$. 
The singular set~$\cV^s$ consists of all of the accumulation points of
vertices: that is,
\[\cV^s = \overline{\bigcup_{i,j}\{x_{i,j}, y_{i,j}\}}.\]
$\pi\vert_{E_0}$ is an isometric embedding of $E_0$ into~$G$, and
$\cV^s$ is a standard middle-thirds Cantor set in~$\pi(E_0)$. 

Let~$\Lambda=\cV^s$. A similar argument to that of
Example~\ref{ex:simple} can be used to show that the complex structure
on $S\setminus\cV^s$ extends uniquely across~$\Lambda$: if $1/(2\cdot
3^{k+2}) < r \le 1/(2\cdot 3^{k+1})$ for $k\ge 0$, then $B_G\lr$ has
$2^k$ connected components, each of which has measure bounded above by
$1/3^k + 4r$ and at most three boundary points in~$G$. Therefore, for
all~$q\in\Lambda$, 
\[\cmqr + r\cdot \cnqr \le \frac{1}{3^k} + 7r \qquad
\left(\frac{1}{2\cdot 3^{k+2}} < r \le \frac{1}{2\cdot 3^{k+1}}
\right),\] so that
\[\int_{1/2\cdot3^{k+2}}^{1/2\cdot3^{k+1}} \frac
{\rmd r}
{\cmqr +r\cdot\cnqr} \ge \frac{\ln(39/25)}{7}\]
for all $k\ge 0$, and~(\ref{eq:generaldivint1}) follows.

\end{example}

\begin{rmk}
In both of the examples, the identifications respect the horizontal
and vertical structure of the square~$P$. Therefore the quadratic
differential $\mathrm{d}z^2$ on~$P$ projects to an $L^1$ quadratic differential
on the quotient (with $L^1$ norm $1$) which is holomorphic away from
the singular set and has simple poles at the fold points.
\end{rmk}

\subsection{Proof of Theorem~\ref{thm:generalmain}}
\label{subsec:proof-generalmain}

\subsubsection{Preliminaries}
In this section the technical tools needed for the proof are
developed: the main focus is on the way in which the number and
structure of the components of $B_G(\Lambda;r)$ vary with~$r$.

\begin{defn}[Injectivity radius]
\label{defn:inj-radius}
A number~$\br>0$ is called an {\em injectivity radius} for~$\Lambda$
if the closure $\ol\CBqr$ of $\CBqr$ is a dendrite, and a
proper subset of~$G$, for every~$q\in\Lambda$ and $r\in(0,\br)$.
\end{defn}

\begin{lem}
\label{lem:injectivity-radius}
An injectivity radius for~$\Lambda$ exists.
\end{lem}
\begin{proof}
By Theorem~\ref{thm:top-structure}, $G$ can be written as
$G=C\cup\Gamma$, where~$C$ is a finite connected graph and~$\Gamma$ is
a union of countably many dendrites, each attached to~$C$ at a single
point, which is an element of~$\ol\cV$.

If~$C\not=\emptyset$ then let $e_1,\ldots,e_N$ denote the edges
of~$C$, and choose arcs $\gamma_i\sbs e_i$ which are disjoint
from~$\ol\cV$ (this is possible since~$\ol\cV$ is a closed and totally
disconnected subset of~$G$). If~$C=\emptyset$, then let~$N=1$ and
$\gamma_1$ be any arc in~$G\setminus\ol\cV$.

Each connected component of $G\setminus\bigcup_{i=1}^N \interior{\gamma_i}$
is a continuum in which any two points are separated by a third point,
and is therefore a dendrite by Remarks~\ref{rmk:dendrites}~a).

Let~$\br =\min_{1\le i\le N}d_G(\gamma_i, \ol\cV)>0$. Then if
$r\in (0,\br)$ and $q\in\Lambda$, the set $\ol\CBqr$ is a subcontinuum
of the component of $G\setminus\bigcup_{i=1}^N\interior{\gamma_i}$
containing~$q$, and is therefore a dendrite by
Remarks~\ref{rmk:dendrites}~c) as required.
\end{proof}

Throughout the remainder of the section, $\br>0$ denotes a fixed choice
of injectivity radius for~$\Lambda$.

\begin{defns}[$\nccr$, $\NCL$, $\CLqr$]
For~$r\in(0,\br)$, let~$\nccr$ denote the number of
connected components of~$B_G\lr$. Let~$\NCL\sbs(0,\br)$ denote the set
of values~$r>0$ at which~$\nccr$ is discontinuous (i.e.\ not locally
constant).

Given~$q\in\Lambda$ and $r\in(0,\br)$,
write~$\CLqr=\Lambda\cap\CBqr$, the set of points of~$\Lambda$ in the
connected component of~$B_G\lr$ which contains~$q$.
\end{defns}

\begin{rmk}
\label{rmk:CLambda}
Fix~$r\in(0,\br)$. Each $\CLqr$ is an open subset of~$\Lambda$ since
$\CBqr$ is open in~$G$. If~$q_1,q_2\in\Lambda$, then $\CL(q_1;r)$ and
$\CL(q_2;r)$ are either equal or disjoint. The sets $\CLqr$ with
$q\in\Lambda$ therefore constitute a partition of~$\Lambda$ by clopen
subsets: in particular, each $\CLqr$ is compact.
\end{rmk}

\begin{lem}\mbox{}
\label{lem:ncc}
\begin{enumerate}[a)]
\item $\nccr$ is finite, and is a decreasing function of~$r$.
\item $\NCL$ is finite if~$\Lambda$ is finite, and otherwise consists
  of a decreasing sequence converging to~$0$.
\item Let~$(r_1,r_2)$ be a complementary component of~$\NCL$ and
  let~$q\in\Lambda$. Then the set $\CLqr$ is constant for
  $r\in(r_1,r_2)$.
\item $\CCqr$ is the set of points of~$G$ which are exactly
  distance~$r$ from~$\CLqr$.
\end{enumerate}
\end{lem}
\begin{proof}\mbox{}
\begin{enumerate}[a)]
\item Let~$K$ be a component of $B_G\lr$. It will be shown that
  $\meas{K}\ge 2r$, which establishes that $\nccr$ is finite, since
  $\meas{G}$ is finite.

First observe that $K$ must contain some point~$q$ of~$\Lambda$. For
let~$x$ be any point of~$K$, so that $d_G(x,q)<r$ for
some~$q\in\Lambda$. Since~$d_G$ is strictly intrinsic
(Lemma~\ref{lem:intrinsic-quotient is intrinsic}), there is an arc
in~$G$ from~$q$ to~$x$ of length $d_G(x,q)$, and this arc is
contained in $B_G\lr$: hence $q\in K$ as required.

Since $r<\br$, there is some point~$y\in G\setminus B_G\qr$. As~$d_G$
is strictly intrinsic, there is an arc in~$G$ from~$q$ to~$y$ of
length at least~$r$. The initial length~$r$ subarc of this arc has
interior contained in~$K$, so that $\meas{K}\ge 2r$ by
Lemma~\ref{lem:metricstructureofG}~d) as required.

If $r<s$ then $B_G\lr\sbs B_G(\Lambda;s)$, so that each connected
component of~$B_G\lr$ is contained in a connected component
of~$B_G(\Lambda;s)$. Every component of~$B_G(\Lambda;s)$
contains a point of~$\Lambda$, and hence contains a component of
$B_G\lr$. Therefore~$\nccr$ is decreasing.

\item On any interval bounded away from~$0$, $r\mapsto\nccr$ is therefore a
bounded decreasing positive integer-valued function, and so has only
finitely many discontinuities.

If~$\Lambda$ is finite then $\nccr=\#\Lambda$ for all sufficiently
small~$r$, so that $\NCL$ is finite.

If~$\Lambda$ is infinite then for every $r_0\in(0,\br)$ pick distinct
points~$q_1,q_2\in\Lambda$ with
$d_G(q_1,q_2)<r_0$. Let~$(p_1,p_2)_{\ol\CB(q_1;r_0)}$ be a
complementary component of $\Lambda$ in $[q_1,q_2]_{\ol\CB(q_1;r_0)}$
(cf. Remark~\ref{rmk:dendrites}~b)), and let~$s=d_G(p_1,p_2)$. Then
$\nccr$ has a discontinuity at $r=s/2<r_0$. Hence $\NCL$ is infinite,
and therefore consists of a sequence converging to~$0$.

\item Let~$r_1<r<s<r_2$. As in the proof of~a), every connected
  component of $B_G(\Lambda;s)$ contains a connected component of
  $B_G\lr$. Since the two sets have the same number of connected
  components, every connected component of~$B_G(\Lambda;s)$ contains
  exactly one connected component of~$B_G\lr$: these two components
  must therefore contain the same points of~$\Lambda$.

\item Write~$C=\CLqr$. Let~$x\in G$. If $d_G(x,C)<r$ then all points
  in a neighborhood of $x$ in~$G$ belong to $\CBqr$, and hence~$x$ is
  not in its boundary $\CCqr$; similarly if $d_G(x,C)>r$.

Suppose, then, that $d_G(x,C)=r$, and let~$\gamma$ be an arc of
length~$r$ in~$G$ from~$x$ to a point~$q$
of~$C$. Then~$x\not\in\CBqr$, but points of~$\gamma$ arbitrarily close
to~$x$ do belong to~$\CBqr$, so that~$x$ lies in the boundary $\CCqr$
of $\CBqr$ as required.
\end{enumerate}
\end{proof}

\begin{defn}[Planar radius, $\PRq$]
A number~$r\in(0,\br)$ is said to be a {\em $\ql$-planar radius} if
$r\not\in\NCL$ and all of the points of~$\CCqr$ are planar
points. The set of $\ql$-planar radii is denoted
$\PRq\sbs(0,\br)$. 
\end{defn}

\begin{rmk}
\label{rmk:planar-radius}
By Lemma~\ref{lem:ncc}~c) and~d), if~$(r_1,r_2)$ is a complementary
component of~$\NCL$ and~$q\in\Lambda$, then $\CLqr$ is constant
for~$r\in(r_1,r_2)$, and if~$q'\in \CLqr$
then~$\PR(q')\cap(r_1,r_2)=\PRq\cap(r_1,r_2)$. 
\end{rmk}

\begin{lem}
\label{lem:generalplanarradius}
Let~$q\in\Lambda$. Then $\PRq$ is a full measure open subset
of~$(0,\br)$, and $\cnqr$ is finite and constant on each component of
$\PRq$.
\end{lem}
\begin{proof}
Since $\NCL$ is either finite or a sequence converging to zero, it
suffices to prove the statement in each complementary
component~$(r_1,r_2)$ of~$\NCL$ in~$(0,\br)$. 

For~$r\in(r_1,r_2)$, $C=\CLqr$ is a fixed compact subset of~$G$ with
$\meas{C}=0$, and $\CCqr$ is the set of points of~$G$ which are
distance~$r$ from~$C$ (Lemma~\ref{lem:ncc}). Define $f\co G\to\R$ by
$f(x)=d_G(x,C)$: then the set of $\ql$-non-planar radii in~$(r_1,r_2)$
is given by $(r_1,r_2)\cap f(\ol\cV)$. Since~$f$ is continuous
and~$\ol\cV$ is compact, the set of $\ql$-non-planar radii is a closed
subset of~$(r_1,r_2)$. Moreover, $f$ is distance non-increasing and
$\mu^1_G(\ol\cV)=0$ by Lemma~\ref{lem:metricstructureofG}, so that
$f(\ol\cV)$ has zero 1-dimensional Hausdorff measure (i.e. Lebesgue
measure). Therefore $\PRq\cap(r_1,r_2)$ is a full measure open subset
of $(r_1,r_2)$ as required.

Now let~$r\in\PRq\cap(r_1,r_2)$, and let~$\delta>0$ be small enough
that $(r-\delta, r+\delta)\sbs\PRq$. Then for each~$x\in\CCqr$, the
ball~$I_x=B_G(x;\delta)$ consists entirely of planar points of~$G$,
and is therefore isometric to an interval of length~$2\delta$. Now
$\ol\CBqr$ is a dendrite and~$d_G$ is strictly intrinsic, so~$I_x$
cannot contain any other point of~$\CCqr$. Since $\meas{I_x}=
4\delta$, $\cnqr=\#\CCqr$ is finite as required.

If~$r'\in(r-\delta,r+\delta)$ then each~$I_x$ contains a point of
$\CC(q;r')$ (at distance $|r'-r|$ from~$x$), and hence
$\cn(q;r')\ge\cnqr$. On the other hand, if~$r'\in(r-\delta/4,
r+\delta/4)$, then $(r'-\delta/2, r'+\delta/2)\sbs\PRq$, and by the
same argument $\cnqr\ge\cn(q;r')$. Therefore $\cnqr$ is
locally constant, and hence constant, on each component of $\PRq$ as
required. 
\end{proof}

\subsubsection{Foliated collaring of~$P$}
\label{subsec:cxstruct-foliatedcollar}
Theorems~\ref{thm:removability1} and~\ref{thm:absareazero} will be
used to show that the conformal structure on $S\setminus\cV^s$
extends across~$\Lambda$. The system of annuli required by the
conditions of Theorem~\ref{thm:absareazero} will be constructed in the
projection~$Q=\pi(\tQ)\sbs S$ of a foliated collar~$\tQ$ of~$P$, which
is described in this section. For simplicity of notation it will be
assumed that~$P$ has a single component: if~$P$ has several
components, the collar can be constructed in each of them
independently. It will also be assumed that the paper-folding scheme
is plain, so that the paper surface is a topological sphere: since all
of the constructions of annuli are carried out locally on a scale
smaller than the injectivity radius~$\br$, only minor modifications
are needed in the non-plain case. The construction here is identical
to that of~\cite{O1}, but is included for completeness.

$\tQ$ is constructed as a union of trapezoids whose bases are the
sides of~$P$; whose vertical sides bisect the angles at the vertices
of~$P$; and which have fixed height~$\bh$, chosen small enough that
the trapezoids are far from degenenerate and intersect only along
their vertical sides. It has a {\em horizontal foliation} by leaves
parallel to~$\partial P$, and a {\em vertical foliation} by leaves
joining the base and the top of each trapezoid: see
Figure~\ref{fig:collar}.

\begin{figure}[htbp]
\lab{e0}{\te_0}{}\lab{X}{\te'_0}{}\lab{e1}{\te_1}{}\lab{ei}{\te_{i}}{}
\lab{en1}{\te_{n-1}}{}
\lab{v0}{\tv_0}{}\lab{v1}{\tv_1}{}\lab{v2}{\tv_2}{}\lab{vn1}{\tv_{n-1}}{}
\lab{vi1}{\tv_{i}}{}\lab{vi}{\tv_{i+1}}{}
\lab{P}{P}{}
\lab{Q}{\tQ_0}{}
\lab{ti1}{\ttheta_{i}}{l}
\lab{Qi}{\tQ_i}{r}
\lab{x}{\tx}{}\lab{verx}{\tver(\tx) = \tver(t)}{l}\lab{gxh}{\tgamma(t,h)}{l}
\lab{h}{h}{}\lab{bh}{\bh}{}\lab{Y}{\te'_i}{}\lab{eh}{\te_i(h)}{l}
\lab{t1}{t_i}{}
\lab{tx}{t}{}
\lab{t2}{t_{i+1}}{}
\lab{...}{\ldots}{}
\begin{center}
\pichere{1.0}{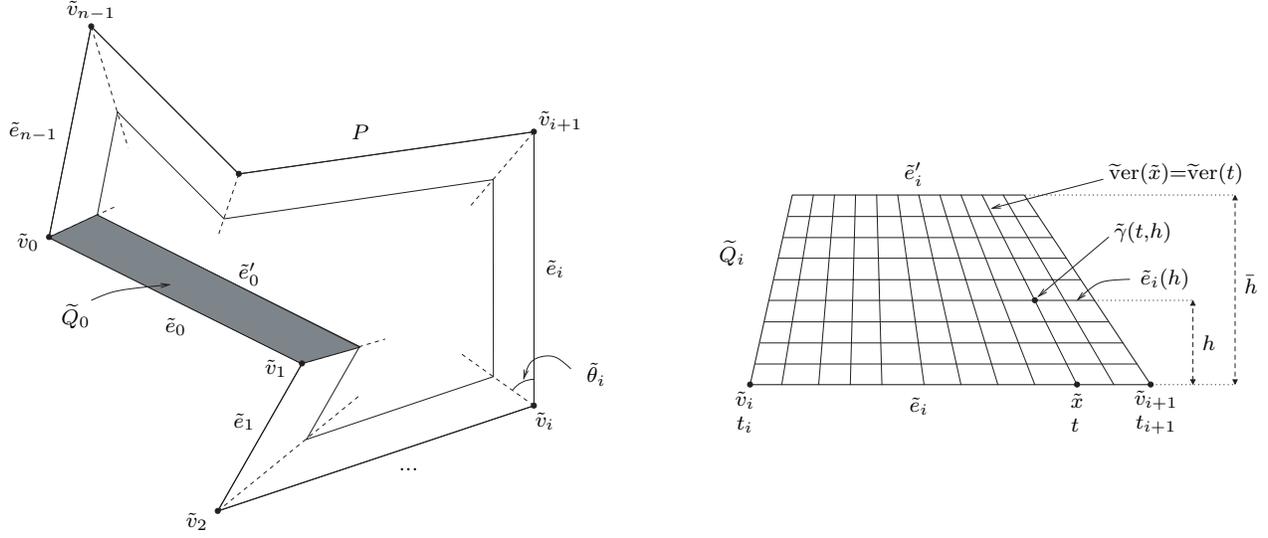}
\end{center}
\caption{The collar~$\tQ$ and its foliations}
\label{fig:collar}
\end{figure}

Choose a labeling~$\tv_i$ ($i=0,\ldots,n-1$) of the vertices of~$P$
listed counterclockwise around~$\partial P$, and let~$\te_i$ be the
side of~$P$ with endpoints~$\tv_i$ and~$\tv_{i+1}$ (here and
throughout, subscripts on cyclically ordered objects are taken~$\mod
n$). Write~$L=|\partial P|$, and
let~$\tgamma_0\co[0,L)\to\partial P$ be the order-preserving
parameterization of~$\partial P$ by arc-length with
$\tgamma_0(0)=\tv_0$. Denote by~$t_i\in[0,L)$ the parameter with
$\tgamma_0(t_i)=\tv_i$. 

A {\em trapezoid} is a quadrilateral in~$\C$ with two parallel
sides, which are called its {\em base} and its {\em top}: the other
sides are called {\em vertical sides}. The {\em height} of the
trapezoid is the distance between the parallel lines containing its
base and its top.

Pick a height~$\bh$ small enough that the trapezoids~$\tQ_i$ which
have bases~$\te_i$, heights~$\bh$, and vertical sides along the rays
bisecting the internal angles of~$P$ satisfy:
\begin{enumerate}[a)]
\item The lengths of the tops of the trapezoids are between half and
  twice the lengths of their bases; and
\item The trapezoids intersect only along their vertical sides.
\end{enumerate}
 This height~$\bh$ is an important quantity in the construction, and
 will remain fixed throughout the remainder of the section.  Denote
 the top of~$\tQ_i$ by $\te_i'$, and let~$\ttheta_i$ be half of the
 internal angle of~$\partial P$ at~$\tv_i$: thus the internal angles
 of~$\tQ_i$ at the endpoints of its base are $\ttheta_i$ and
 $\ttheta_{i+1}$.

Let \[\tQ = \bigcup_{i=0}^{n-1} \tQ_i,\]
a closed collar neighborhood of~$\partial P$ in~$P$.

For each $h\in[0,\bh]$, let $\te_i(h)\subset\tQ_i$ be the segment
parallel to the base of~$\tQ_i$ at height~$h$, so that
$\te_i(0)=\te_i$ and $\te_i(\bh)=\te_i'$. Then, for each~$h$, the
union $\thor(h)$ of these segments is a polygonal simple closed curve:
these simple closed curves are the leaves of the {\em horizontal
  foliation}
\[\tHor = \left\{\thor(h)\,:\, h\in[0,\bh]\right\}\]
of~$\tQ$. The parameter~$h$ is called the {\em height} of the leaf
$\thor(h)$. 

Write
\[\tQ(h) = \bigcup_{h'\in[0,h]}\thor(h'),\] 
the subset of~$\tQ$ consisting of leaves with heights not
exceeding~$h$: therefore $\tQ(h)\subset \tQ = \tQ(\bh)$ is also a
closed collar neighborhood of~$\partial P$ for each~$h\in(0,\bh]$.

To construct the vertical foliation, let~$\varphi_i\co\te_i\to\te_i'$
be the orientation-preserving scaling from~$\te_i$ to
$\te_i'$. For each $\tx = \tgamma_0(t) \in \te_i$, denote by
$\tver(\tx)$ or $\tver(t)$ the straight line segment which joins $\tx$
to $\varphi_i(\tx)$. These segments are the leaves of the {\em
  vertical foliation}
\[\tVer = \left\{\tver(t)\,:\, t\in[0,L)\right\}\]
of $\tQ$. 

Define $\ttheta\co[0,L)\setminus\{t_0,\ldots,t_{n-1}\} \to(0,\pi)$ by
  setting $\ttheta(t)$ to be the angle between $\partial P$ and
  $\tver(t)$ at $\tgamma_0(t)$: that is, the angle between the
  oriented side of $\partial P$ containing $\tgamma_0(t)$ and the leaf
  $\tver(t)$ pointing into~$P$. This function has well-defined limits
  as~$t$ approaches each~$t_i$ from the left or the right:
  $\ttheta(t_i^-)=\pi-\ttheta_i$, and $\ttheta(t_i^+) =
  \ttheta_i$. The notation $\ttheta(\tx)=\ttheta(t)$ will also be used
  when $\tx = \tgamma_0(t)$.

The foliations $\tHor$ and $\tVer$ yield a parameterization
\[\tgamma\co [0,L) \times [0,\bh] \to \tQ\]
of $\tQ$, where $\tgamma(t,h)$ is the unique point of $\tver(t) \cap
\thor(h)$.

For each~$h\in(0,\bh]$, denote by $\tpsi_h\co \tQ(h)\to\partial P$ the
retraction of $\tQ(h)$ onto $\partial P$ which slides each point along
its vertical leaf:
\[
\tpsi_h(\gamma(t,h')) = \gamma(t,0) \qquad \text{(all $t\in[0,L)$ and
    $h'\in[0,h]$)}.
\]
In particular, $\tpsi_{\bh}$ is a retraction which squashes all
of~$\tQ$ onto~$\partial P$.

The projections to the paper surface~$S$ of the structures defined
above are denoted by removing tildes. Thus $Q = \pi(\tQ)$ is a closed
disk neighborhood of the scar~$G$, which is a dendrite by the
plainness hypothesis and Theorem~\ref{thm:plain-top-structure}; and
similarly $Q(h) = \pi(\tQ(h))$ is a closed subdisk neighborhood for
each $h\in(0,\bh]$.  $Q$~has horizontal and vertical foliations $\Hor
= \pi(\tHor)$ and $\Ver = \pi(\tVer)$. The leaves of~$\Hor$ are
projections of leaves of~$\tHor$, $\hor(h) = \pi(\thor(h))$: these are
topological circles except for~$\hor(0)=G$. The leaves of $\Ver$,
however, are unions of projections of leaves of~$\tVer$: for
each~$x\in G$, the leaf of~$\Ver$ containing~$x$ is defined to be
\[\ver(x) := \bigcup \left\{ \pi(\tver(\tx))\,:\,\tx\in\pi^{-1}(x)
\right\}. \]
Thus $\ver(x)$ is an arc if and only if $\#\pi^{-1}(x)\le 2$. If~$x$ is
a $k$-vertex for $k>2$ then $\ver(x)$ is a star with~$k$
branches.

The disks~$Q(h)$ for $0<h\le \bh$ are similarly foliated by horizontal
leaves $\hor(h')$ with $0< h'\le h$, and vertical leaves~$\ver_h(x)$,
which are the leaves $\ver(x)$ trimmed at their intersection with
$\hor(h)$.

The composition $\gamma := \pi\circ\tgamma\co [0,L)\times[0,\bh]\to Q$
  parameterizes~$Q$, although it is not injective on the preimage
  of~$G$. 

Because the retractions $\tpsi_h\co\tQ(h)\to\partial P$ fix $\partial
P$ pointwise, the compositions
\[\psi_h := \pi \circ \tpsi_h \circ \pi^{-1} \co Q(h)\to G\]
are well-defined retractions of $Q(h)$ onto $G$.

\subsubsection{The system of annuli $\AnnL\qrs$}
\label{subsec:cxstruct-annuli}
 In this section annuli $\AnnL\qrs$ about~$q$ will be constructed for
 each pair of $\ql$-planar radii $r<s$. The annuli will be defined as
 differences of two topological closed disks: $\AnnL\qrs =
 \Int(\DL\qs) \setminus \DL\qr$.

The parameter~$r$ in the disk $\DL\qr$ is related to the trimmed
vertical leaves which it contains: $\psi_{\bh}(\DL\qr)=\ol\CBqr$. A
second parameter~$h$ could be used to specify the horizontal leaves
which intersect~$\DL\qr$, but it is convenient to make this parameter
dependent on~$r$, via the function $h\co[0,\br]\to[0,\bh/2]$ defined
by
\[h(r) := \left(\frac{\bh}{2\br}\right)\,r.\]

\begin{defn}[$\DL\qr$]
\label{defn:Dqr}
Let~$q\in\Lambda$ and $r\in (0,\br)$. The subset $\DL\qr$ of~$Q$ is
defined by
\[\DL\qr := \psi_{h(r)}^{-1}\left(\ol\CBqr\right).\]
\end{defn}

Alternatively, $\DL\qr$ is the intersection of~$Q(h(r))$ with the
union of the vertical leaves $\ver(x)$ with $x\in\ol\CBqr$. 

\begin{rmk}
\label{rmk:D-intersection}
For all~$q\in\Lambda$, $\bigcap_{r\in(0,\br)}\DL\qr =
\bigcap_{r\in(0,\br)}\ol\CBqr=\{q\}$. For let~$\veps>0$, and
let~$(r_1,r_2)\sbs\PRq\cap(0,\veps)$, where $0<r_1<r_2<\veps$. Then
$\CB(q;(r_2-r_1)/2)$ does not contain any point~$x$
with~$d_G(q,x)>\veps$. ($\frac{r_2-r_1}{2}$-balls about points
of~$\Lambda$ further than~$r_2$ from~$q$ cannot intersect those about
points closer than~$r_1$ from~$q$.)
\end{rmk}

\begin{lem}
\label{lem:dqr-is-disk}
Let~$q\in \Lambda$ and $r\in(0,\br)$ be a $\ql$-planar radius. Write
  $n=\cnqr$ and $h=h(r)$. Then $\DL\qr$ is a topological closed
  disk, whose boundary is composed of~$n$ disjoint subarcs of the
  horizontal leaf $\hor(h)$, and the~$n$ trimmed vertical
  leaves~$\ver_h(x)$ with $x\in \CCqr$.
\end{lem}

\begin{figure}[htbp]
\lab{0}{x_0}{}
\lab{1}{x_1}{}
\lab{2}{x_2}{}
\lab{3}{x_3}{}
\lab{4}{x_4}{}
\lab{5}{q}{}
\lab{6}{\ver_h(x_1)}{}
\lab{7}{\CBqr}{}
\lab{8}{\DL\qr}{}
\begin{center}
\pichere{0.8}{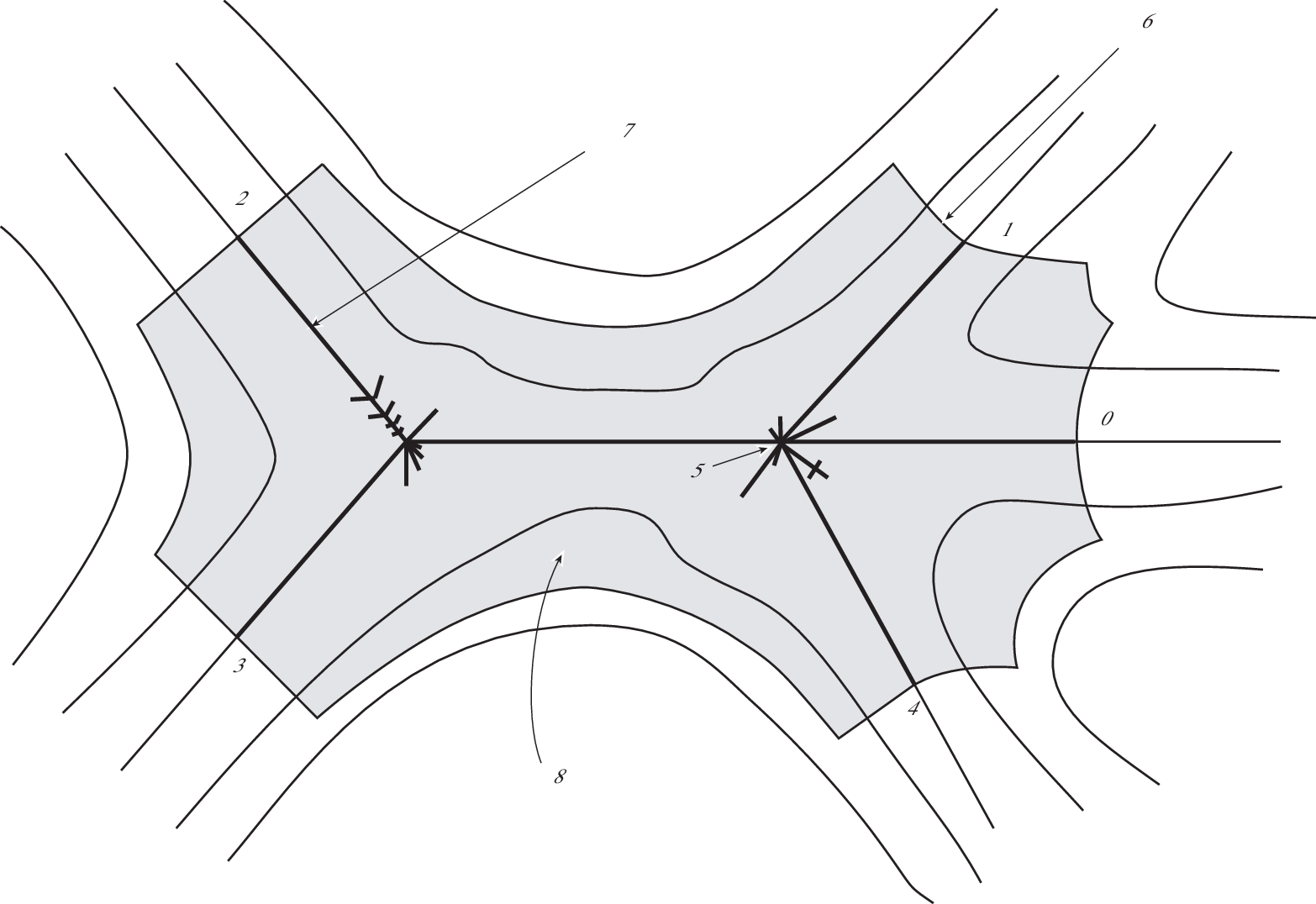}
\end{center}
\caption{The disk $\DL\qr$}
\label{fig:disk-proof}
\end{figure}

\begin{proof} 
(See Figure~\ref{fig:disk-proof}.) $n$ is finite by
  Lemma~\ref{lem:generalplanarradius}.
  Write $\CCqr=\{x_0,\ldots,x_{n-1}\}$. The first step is to show that
\[\partial_{Q(h)} \DL\qr = \psi_h^{-1}\left(\CCqr\right) = 
\bigcup_{i=0}^{n-1} \psi_h^{-1}(x_i) = \bigcup_{i=0}^{n-1}\ver_h(x_i).\]
  
 That $\partial_{Q(h)}\DL\qr \sbs \psi_h\I(\CCqr)$ is immediate from
 the definition of~$\DL\qr$ (recall that
 $\DL\qr=\psi_h^{-1}\left(\ol\CBqr)\right)$ and $\CCqr =
 \partial_G\CBqr$). The opposite inclusion follows from the fact that
 $r$ is $(q,\Lambda)$-planar, so that every neighborhood of each
 point of $\CCqr$ contains both points wich are closer to~$q$ and
 points which are further away (cf. the proof of
 Lemma~\ref{lem:generalplanarradius}).

Since the points~$x_i$ are planar, each $\psi_h^{-1}(x_i) =
\ver_h(x_i)$ is an arc which intersects~$\hor(h) = \partial_S Q(h)$
exactly at its endpoints: that is, a cross cut in~$Q(h)$. For
each~$i$, $Q(h)\setminus\ver_h(x_i)$ has exactly two components, one
of which intersects~$G$ in the complement of $\ol\CBqr$. Therefore
every $\ver_h(x_j)$ with $j\not=i$ is contained in the same component
as~$q$. It follows that $\partial_S\DL\qr$ is the simple closed curve
composed of the arcs $\ver_h(x_i)$ and the~$n$ subarcs of~$\hor(h)$
joining the endpoints of consecutive cross cuts in the cyclic order
around~$\hor(h)$.
\end{proof}

The annuli which will be used in the proof of Theorem~\ref{thm:generalmain}
can now be defined.

\begin{defn}[$\AnnL\qrs$]
Let $q\in \Lambda$, and $r,s\in\PRq$ with $r<s$. The
  subset $\AnnL\qrs$ of~$Q$ is defined by
\[\AnnL\qrs = \Int(\DL\qs) \setminus \DL\qr.\] 
\end{defn}

By Lemma~\ref{lem:dqr-is-disk}, and since $\DL\qr\subset\Int(\DL\qs)$,
$\AnnL\qrs$ is an open annular region with~$q$ in its bounded
complementary component (the complementary component not containing
$\partial Q$).

The following definition and theorem provide a lower bound on the
modules of these annuli.

\begin{defn}[Paper-folding goodness function]
\label{defn:general-pfgf}
Let~$(P,\cP)$ be a surface paper-folding scheme and~$\Lambda$ be a
clopen subset of~$\cV^s$. Define
$\iota_\Lambda\co\Lambda\times(0,\br)\to[0,\infty)$ by
\[
\iota_\Lambda\qr = 
\begin{cases}
\dfrac{M}{\cmqr + r\cdot\cnqr} & \text{ if }\cnqr < \infty, \\ \\
0 & \text{ otherwise,}
\end{cases}
\]
where $M=\frac15\min(\br/\bh, \bh/\br)$. $\iota_\Lambda$ is called a {\em
  paper-folding goodness function} for~$(P,\cP)$ on~$\Lambda$. 
\end{defn}

\begin{rmk}
By Lemma~\ref{lem:generalplanarradius}, the set of radii~$r\in(0,\br)$
at which $\iota_\Lambda\qr=0$ or $\iota_\Lambda\qr$ is discontinuous
is disjoint from~$\PRq$, and therefore has measure zero.
\end{rmk}

\begin{rmk}
\label{rmk:upper-bound}
Since~$\cmqr=\meas{\CBqr}\ge\meas{B_G\qr}\ge 2r$ for $r\in(0,\br)$,
$\iota\SL\qr$ is bounded above by $M/2r$.
\end{rmk}

\begin{thm}
\label{thm:module-bound}
Let~$q\in\Lambda$ and~$r,s\in\PRq$ with $r<s$. Then
\[\mod\AnnL\qrs \ge \int_r^s \iota_\Lambda(q;t)\rmd t.\]
\end{thm}

\begin{proof}[Sketch proof]
The proof is analogous to that in~\cite{O1} for the case
where~$\Lambda$ consists of a single isolated singularity: a sketch of
the main ideas is given here.

Let~$(r_1,r_2)$ be contained in the full measure open subset~$\PRq$
of~$(0,\br)$. As~$r$ increases through~$(r_1,r_2)$, the
set~$\CBqr$ changes regularly: if $r_1<r<s<r_2$, then
$\CB\qs\setminus\CBqr$ is a union of~$n=\cnqr$ disjoint intervals of
length~$s-r$ which are disjoint from~$\ol\cV$.

It follows that the disk~$\DL\qr$ also changes regularly
for~$r\in(r_1,r_2)$, and using the foliations of~$Q$ it is possible to
obtain an upper bound for the rate of change of the Euclidean area of
$\DL\qr$, 
\[\frac{\rmd}{\rmd r} \Area_S(\DL\qr) \le \frac{5\bh}{4\br}\left(
\cmqr+r\cdot\cnqr\right) .\] 

Similarly, the distance between the boundary components of $\AnnL\qrs$
can be estimated for $r_1<r<s<r_2$:
\[d_S\left(\partial\DL\qr, \partial\DL\qs\right) \ge
C(s-r),\]
where $C=\frac{1}{2}\min(\bh/\br,1)$.

Now if $r_1=s_0<s_1<\cdots<s_k=r_2$ is a partition of~$[r_1,r_2]$, it
follows from Theorem~\ref{thm:additivity-module} and~(\ref{eq:modest}) that
\begin{eqnarray*}
\mod\AnnL(q;r_1,r_2) &\ge&
\sum_{j=1}^k \mod\AnnL(q;s_{j-1},s_j) \\
&\ge& \sum_{j=1}^k \frac{C^2\,(s_j-s_{j-1})^2}
{\Area_S(\DL(q;s_j)) - \Area_S(\DL(q;s_{j-1}))}.
\end{eqnarray*}
Taking a limit over increasingly fine partitions gives
\[\mod\AnnL(q;r_1,r_2) \ge \int_{r_1}^{r_2} 
\frac{C^2\rmd t}{\frac{\rmd}{\rmd t}\Area_S(\DL(q;t))} \ge
\int_{r_1}^{r_2}\iota_\Lambda(q;t)\rmd t.\] 

That the same bound holds for arbitrary~$r_1,r_2\in\PRq$ follows from
additivity of the module and the fact that~$\PRq$ is a full measure
open subset of~$(0,\br)$.
\end{proof}

\subsubsection{Proof of Theorem~\ref{thm:generalmain}}
 By Lemma~\ref{lem:ncc}~b), $\NCL=\{r_k\,:\, k\ge 1\}$ where~$(r_k)$
 is a decreasing sequence converging to~$0$ if~$\Lambda$ is infinite,
 and is finite if $\Lambda$ is finite: in the finite case, extend
 $\NCL$ arbitrarily to a set of the form $\{r_k\,:\, k\ge 1\}$,
 where~$(r_k)$ is a decreasing sequence in~$(0,\br)$ converging
 to~$0$.

Let~$k\ge1$. By Remark~\ref{rmk:CLambda} and Lemma~\ref{lem:ncc}~c),
there is an equivalence relation~$\sim_k$ on~$\Lambda$ whose
classes~$\erk{q}$ are equal to $\CLqr$ for all~$r\in(r_{k+1}, r_k)$:
there are~$n_k=\ncc(r_k)$ such classes. The
set~$\PRq\cap(r_{k+1},r_k)$ of $\ql$-planar radii in~$(r_{k+1},r_k)$
is a full measure open subset of~$(r_{k+1},r_k)$ which is well defined
on equivalence classes, by Remark~\ref{rmk:planar-radius} and
Lemma~\ref{lem:generalplanarradius}.

The idea of the proof is to obtain lower bounds on the modules of the
annuli $\AnnL(q;r_{k+1}, r_k)$, but since $r_{k+1}$ and $r_k$ are not
$(q,\Lambda)$-planar radii, it is necessary first to increase
$r_{k+1}$ and decrease $r_k$ by a controlled
amount~$\veps_k(\erk{q})>0$. This number is chosen so that both
$r_{k+1}+\veps_k(\erk{q})$ and $r_k-\veps_k(\erk{q})$ belong
to~$\PRq$, and small enough that
\[\veps_k(\erk{q}) \le \min
\left(
\frac{r_{k+1}}{2^{k-1} M},\, \frac{r_k-r_{k+1}}{3}
\right).
\]

 Define an annulus $\Ann_k(\erk{q})$ for each
 $\erk{q}\in\Lambda/\!\sim_k$ by
\[\Ann_k(\erk{q}) = \AnnL(q; r_{k+1}+\veps_k(\erk{q}),
r_k-\veps_k(\erk{q})) \]
(see Figure~\ref{fig:nested}).

\begin{figure}[htbp]
\lab{1}{q_1}{}
\lab{2}{q_2}{}
\lab{3}{q_3}{}
\lab{4}{q_4}{}
\lab{5}{q_5}{}
\lab{6}{q_6}{}
\lab{7}{\Ann_1([q_i]_1), i=1,2,3}{}
\lab{8}{\Ann_2([q_i]_2), i=4,5}{}
\lab{9}{\Ann_2([q_6]_2)}{}
\begin{center}
\pichere{0.65}{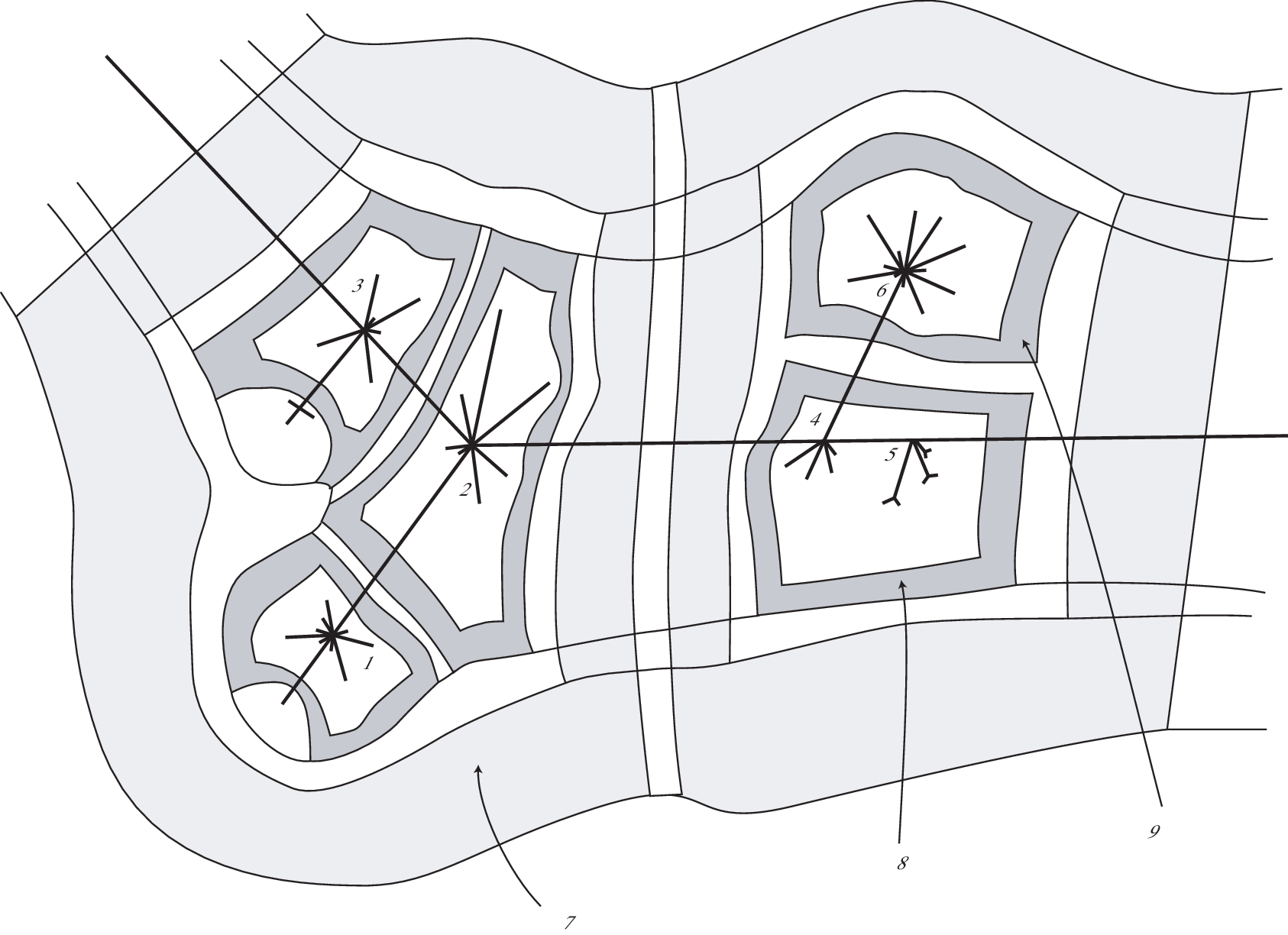}
\end{center}
\caption{The annuli~$\Ann_k(\erk{q})$}
\label{fig:nested}
\end{figure}

It follows from Theorem~\ref{thm:module-bound} and
Remark~\ref{rmk:upper-bound} that
\begin{eqnarray*}
\mod\Ann_k(\erk{q}) &\ge&
\int_{r_{k+1}+\veps_k(\erk{q})}^{r_k-\veps_k(\erk{q})}
\iota_\Lambda(q;t)\,\rmd t\\
&\ge& \int_{r_{k+1}}^{r_k} \iota_\Lambda(q;t)\,\rmd t -
\frac{2\veps_k(\erk{q})M}{2r_{k+1}} \\
&\ge& \int_{r_{k+1}}^{r_k} \iota_\Lambda(q;t)\,\rmd t - \frac{1}{2^{k-1}}.
\end{eqnarray*}
In particular, if~(\ref{eq:generaldivint1}) holds then $\sum_{k\ge 1}
\mod\Ann_k(\erk{q})$ diverges for all~$q\in\Lambda$. 

Observe the following properties of these annuli:
\begin{enumerate}[a)]
\item The~$n_k$ annuli~$\Ann_k(\erk{q})$ are mutually disjoint, since
  \[\Ann_k(\erk{q})\sbs \DL(q; r_k-\veps(\erk{q})) \sbs
  \psi\I_{h(r_k)}\left(\CB(q;r_k) \right),\]
and $\CB(q_1;r_k)\cap\CB(q_2;r_k)=\emptyset$ unless $q_1\sim_k q_2$.

\item All of the points of~$\erk{q}$ are contained in the bounded
  complementary component $\DL(q;r_{k+1}+\veps(\erk{q}))$
  of~$\Ann_k(\erk{q})$, and all of the points of
  $\Lambda\setminus\erk{q}$ are contained in the unbounded
  complementary component.
\item $\Ann_{k+1}(\er{q}{k+1})$ is nested into $\Ann_k(\erk{q})$ for
  each~$q\in\Lambda$ and $k\ge 1$. 
\item Let~$(q_k)$ be any sequence in~$\Lambda$ with the property that
  $\Ann_k(\erk{q_k})$ is a nested sequence of annuli. Then there is a
  unique~$q\in\Lambda$ such that $\Ann_k(\erk{q_k})=\Ann_k(\erk{q})$
  for all~$k$. For the sequence of bounded complementary components
  contains at least one point~$q\in S$ in its intersection. This point
  must lie in~$G$ since the sequence of inner heights of the annuli
  converges to zero; and it must lie in~$\Lambda$ since~$\Lambda$ is a
  closed subset of~$G$ and the sequence of inner radii of the annuli
  converges to zero. Then since~$q\in\Lambda$ lies in the bounded
  complementary component of every $\Ann_k(\erk{q_k})$, it follows
  from~b) above that $q\sim_k q_k$ for all~$k$.
\end{enumerate}

The proof of Theorem~\ref{thm:generalmain} can now be completed.
Let~$q\in\Lambda$. Since~$\Lambda$ is clopen in~$\cV^s$, there is
some~$r\in\PR(q)$ with $\DL\qr\cap\cV^s\sbs\Lambda$, and it suffices
to show that the conformal structure on~$S\setminus\cV^s$ extends
uniquely over $\DL\qr$.

Let~$W=\DL\qr\setminus\Lambda$. Then~$W$ is a planar Riemann surface,
and so by Koebe's General Uniformization Theorem there is a
uniformizing isomorphism $\phi\co W\to\phi(W)$ of~$W$ onto a
domain~$\phi(W)\sbs\C$.

Let~$K$ be such that~$r_K<r$. Pick any~$q'\in \DL\qr\cap\Lambda$, and
write $A_k(q')=\phi(\Ann_k(\er{q'}{k}))$ for each~$k\ge K$. Then
$\sum_{k\ge K} \mod A_k(q')$ diverges by~(\ref{eq:generaldivint1}) and
the argument above, so that the intersection of the bounded
complementary components of the~$A_k(q')$ is a single
point~$z(q')$. Defining $\phi(q')=z(q')$ for each~$q'\in
\DL\qr\cap\Lambda$ extends~$\phi$ to a homeomorphism from~$\DL\qr$
into~$\C$. 

For~$k\ge K$, define $U_k=\bigcup_{q'\in
  \DL\qr\cap\Lambda}A_k(q')$. Then the sets~$U_k$ satisfy the
hypotheses of Theorem~\ref{thm:absareazero} by properties~a) to~d)
above, so that~$\phi(\Lambda)$ is a totally disconnected set of
absolute area zero. Hence if~$\psi\co W\to\psi(W)$ is a second choice
of uniformizing isomorphism on~$W$, then $\psi\circ\phi^{-1}$ extends
uniquely across~$\phi(\Lambda)$ to a conformal map on $\phi(\DL\qr)$ by
Theorem~\ref{thm:removability1}. There is therefore a uniquely defined
conformal structure in $\DL\qr$ which agrees with the given conformal
structure on $\DL\qr\setminus\Lambda$. 
\qed

\section{Modulus of continuity}
\label{sec:modulus-continuity}

Suppose that $(P,\cP)$ is a plain paper-folding scheme with associated
paper surface~$S$, which is a topological sphere by
Theorem~\ref{thm:plain-top-structure}; and that
condition~(\ref{eq:generaldivint1}) holds at every point~$q$
of~$\Lambda=\cV^s$, so that~$S$ has a unique conformal structure which
agrees with the natural conformal structure on $S\setminus\cV^s$. It
follows that~$S$ is conformally isomorphic to the Riemann
sphere~$\csph$. The aim in this section is to construct an explicit
modulus of continuity for a suitably normalized uniformizing map $u\co
S\to\csph$. This modulus of continuity depends only on the geometry
of~$P$ (specifically, on~$|\partial P|$ and on the collaring
height~$\bh$), and on the metric on~$G$, as expressed by
the injectivity radius~$\br$ and the paper-folding goodness function
$\iota\SL$ (and an alternative version of this function for
non-singular points which is described in
Section~\ref{sec:preliminaries}).

The motivation for doing this is that when there is a
uniform modulus of continuity across a family of such paper-folding
schemes, it is possible to use Arzel\`a-Ascoli type arguments to
construct a limiting complex sphere.

After some preliminaries in Section~\ref{sec:preliminaries}, a local
modulus of continuity for~$u$ is obtained at each point of~$S$ (with respect
to the metric~$d_S$ on~$S$ and the spherical metric $d_{\csph}$
on~$\csph$) in Section~\ref{sec:modul-cont-at-point}. In
Section~\ref{sec:glob-modul-cont} it is shown that these local moduli
of continuity patch together to give a global modulus of continuity
for $\phi:=u\circ\pi\co P\to\csph$.

\begin{rmk}
Because the local modulus of continuity derived in
Section~\ref{sec:modul-cont-at-point} arises from a {\em local}
calculation, it applies equally to a uniformizing map from a disk
neighborhood of a point of~$S$ to the unit disk in~$\C$ for an
arbitrary Riemann paper surface. 
\end{rmk}

\subsection{Preliminaries}
\label{sec:preliminaries}

The following assumptions are made throughout Section~\ref{sec:modulus-continuity}:
\begin{itemize}
\item $(P,\cP)$ is a plain paper-folding scheme with associated paper
  sphere~$S$ and scar~$G$.  
\item  $\bh>0$ denotes a fixed choice of collaring height for~$P$
  (Section~\ref{subsec:cxstruct-foliatedcollar}). 
\item $\Lambda=\cV^s$, the singular set of~$G$, and
  condition~(\ref{eq:generaldivint1}) holds at every
  point~$q\in\Lambda$. 
\item $\br>0$ denotes a fixed choice of
  injectivity radius for~$\Lambda$ (Definition~\ref{defn:inj-radius}).
\end{itemize}

In particular, $S$ is conformally isomorphic to the Riemann
sphere~$\csph$. A fixed choice of uniformizing map $u\co S\to\csph$ is
made as follows. Pick points $\tp_0\in\partial P$ and $\tp_\infty\in
P_{\bh}$ (recall that~$P_{\bh}$ is the complement in~$P$ of the
height~$\bh$ collaring~$\tQ(\bh)$ of~$P$). Let~$p_0:=\pi(\tp_0)\in G$
and $p_\infty:=\pi(\tp_\infty)\in S\setminus Q(\bh)$, and define $u\co
S\to\csph$ to be the isomorphism with the following normalization:
\begin{itemize}
\item $u(p_0)=0$;
\item $u(p_\infty)=\infty$; and
\item the reciprocal~$\Phi:=1/\phi$ of the composition
  $\phi:=u\circ\pi\co P\to\csph$ satisfies $\Phi'(\tp_\infty)=1$.
\end{itemize}

\begin{defn}[Modulus of continuity]
\label{defn:mod-cont}
Let~$\rho\co[0,\veps)\to[0,\infty)$, for some~$\veps>0$, be a
    continuous and strictly increasing function with~$\rho(0)=0$. A
    function $f\co(X,d_X)\to(Y,d_Y)$ between metric spaces has {\em
      modulus of continuity}~$\rho$ at $x_0\in X$ if, for every~$x\in
    X$ with $d_X(x_0,x)<\veps$,
\[d_Y(f(x_0),f(x)) \le \rho(d_X(x_0,x)).\]
$f$ is said to have {\em global modulus of continuity~$\rho$} if this
inequality holds for every pair of points $x_0,x\in X$ with
$d_X(x_0,x)<\veps$. 
\end{defn}

\begin{rmk}
\label{rmk:mod-cont-transfer}
Since the projection~$\pi\co P\to S$ is distance non-increasing, a
modulus of continuity $\rho_q$ for $u\co S\to\csph$ at~$q\in S$ is
also a modulus of continuity for the composition $\phi:=u\circ\pi \co
P\to\csph$ at the points of~$\pi^{-1}(q)$.
\end{rmk}

\medskip\medskip

In Section~\ref{sec:conf-struct-paper}, disks~$\DL\qr$,
annuli~$\AnnL\qrs$, and paper-folding goodness functions~$\iota\SL\qr$
providing a lower bound on the modules of the annuli were constructed
depending on the set~$\Lambda$. The only properties of the
set~$\Lambda$ which were used in the constructions were that it is a
closed subset of~$G$ with zero measure. In particular, the
constructions can be repeated with~$\Lambda=\{q\}$, where~$q$ is an
arbitrary {\em non-singular} point of~$G$. The objects thus
constructed are denoted by dropping the subscript~$\Lambda$:
thus~$D\qr$ is a disk containing~$q$ for every $q$-planar
radius~$r\in(0,\br)$; $\Ann\qrs$ is an annular region containing $q$ in its
bounded complementary component for every pair of $q$-planar radii
$0<r<s<\br$; and
\begin{equation}
\label{eq:standard-mod-lower-bound}
\mod \Ann\qrs \ge \int_r^s\iota\qt\,\rmd t,
\end{equation}
where 
\[
\iota\qr = 
\begin{cases}
\dfrac{M}{m\qr + r\cdot n\qr} & \text{ if }n\qr < \infty, \\ \\
0 & \text{ otherwise.}
\end{cases}
\]
Here $m(q;r) = \meas{B_G(q;r)}=\meas{\operatorname{CB}_{\{q\}}\qr}$, and
$n(q;r)=\operatorname{cn}_{\{q\}}(q;r)$ is the number of points of~$G$ at
distance exactly~$r$ from~$q$.

Notice that
\begin{equation}
\label{eq:regular-integral-diverges}
\int_0 \iota(q;s)\,\rmd s = \infty
\end{equation}
at any non-singular point~$q$: for if~$q$ is an isolated valence~$k$
vertex or a planar point ($k=2$), then $m(q;s)=2ks$ and $n(q;s)=k$ for
all sufficiently small~$s$.

\medskip\medskip

The modulus of continuity for~$u$ at points~$q$ in a particular
neighbourhood~$Q(\veps)$ of~$G$ will be obtained by constucting
annular regions~$A$ in~$Q(\bh)$ for which both~$q$ and a nearby
point~$x$ at prescribed distance from~$q$ lie in the same
complementary component. An upper bound on $d_\csph(u(q),u(x))$ can
then be found by combining a lower bound for $\mod A$ (arising, for
example, from Theorem~\ref{thm:module-bound}) with the upper bound
for $\mod u(A)$ provided by Theorem~\ref{thm:GAT}. There are two
issues which need to be addressed in doing this: 
\begin{enumerate}[a)]
\item To find a radius~$R$ with the property that~$u(A)$ is contained
  in the disk~$|z|\le R$, so that Theorem~\ref{thm:GAT} can be
  applied; and
\item To relate the distance~$d_S$ in~$S$ with the parameter~$r$ used
  in the construction of the disks $\DL\qr$ and $D\qr$.
\end{enumerate}

The first of these issues is dealt with by the following lemma, which
is a simple consequence of Koebe's one-quarter theorem and the
normalization of~$u$.
\begin{lem}
\label{lem:Koebe}
There is a constant~$R>0$, depending only on~$\bh$, such that
$u(Q(\bh/2))$ is contained in the disk $|z|<R$ in~$\C$.
\end{lem}
\qed

For the second issue, the following technical lemma, which is proved
in~\cite{O1}, is crucial. The distance of a point~$q\in Q(\bh)$
from~$G$ and from~$\Lambda$ is described by two numbers~$h_q$ and
$d_q$:

\begin{defns}[$\psi$, $h_q$, $d_q$]
\label{defn:h_q and d_q, and psi}
Let~$\psi:=\psi_{\bh}$ denote the retraction of~$Q(\bh)$
onto~$G$. Given~$q\in Q(\bh)$, let~$h_q\ge 0$ denote the {\em height}
of~$q$ (that is, $q$ lies in the horizontal leaf $\hor(h_q)$); and
let~$d_q:=d_G(\psi(q),\Lambda)\ge 0$.
\end{defns}

Notice that both~$h_q$ and~$d_q$ depend continuously on~$q\in
Q(\bh)$. 

\begin{lem}
\label{lem:technical}
Let
\begin{equation}
\label{eq:delta}
\delta = \delta(\br,\bh,|\partial P|) := \frac{1}{4}\min\left\{
\br,\,\bh,\,\frac{2\br\bh}{|\partial P|}
\right\}>0,
\end{equation}
and define $\mu\co Q(\delta)\times[0,\delta)\to[0,\br)$ by
\begin{equation}
\label{eq:mu}
\mu(q;t):=\frac{\br}{2\delta}(t+h_q).
\end{equation}
Then, for every $q\in Q(\delta)$ and every $r\in[0,\delta)$,
\begin{eqnarray*}
\ol{B}_S\qr &\sbs& \DL(\psi(q); \mu(q;r)) \qquad \text{if
  $\psi(q)\in\Lambda$,} \\
\ol{B}_S\qr &\sbs& D(\psi(q); \mu(q;r)) \qquad \text{ if
  $\psi(q)\not\in\Lambda$.} 
\end{eqnarray*}
\end{lem}
\qed

See Figure~\ref{fig:ballindisk}: note that it is clear that some
result of this form must hold, and the content of
Lemma~\ref{lem:technical} is the specific constants obtained.

\begin{figure}[htbp]
\lab{q}{q}{}\lab{pq}{\psi(q)}{}
\lab{B}{\ol{B}_S\qr}{b}
\lab{D}{D\left(\psi(q);\,\mu(q;r)\right)}{t}
\begin{center}
\pichere{0.5}{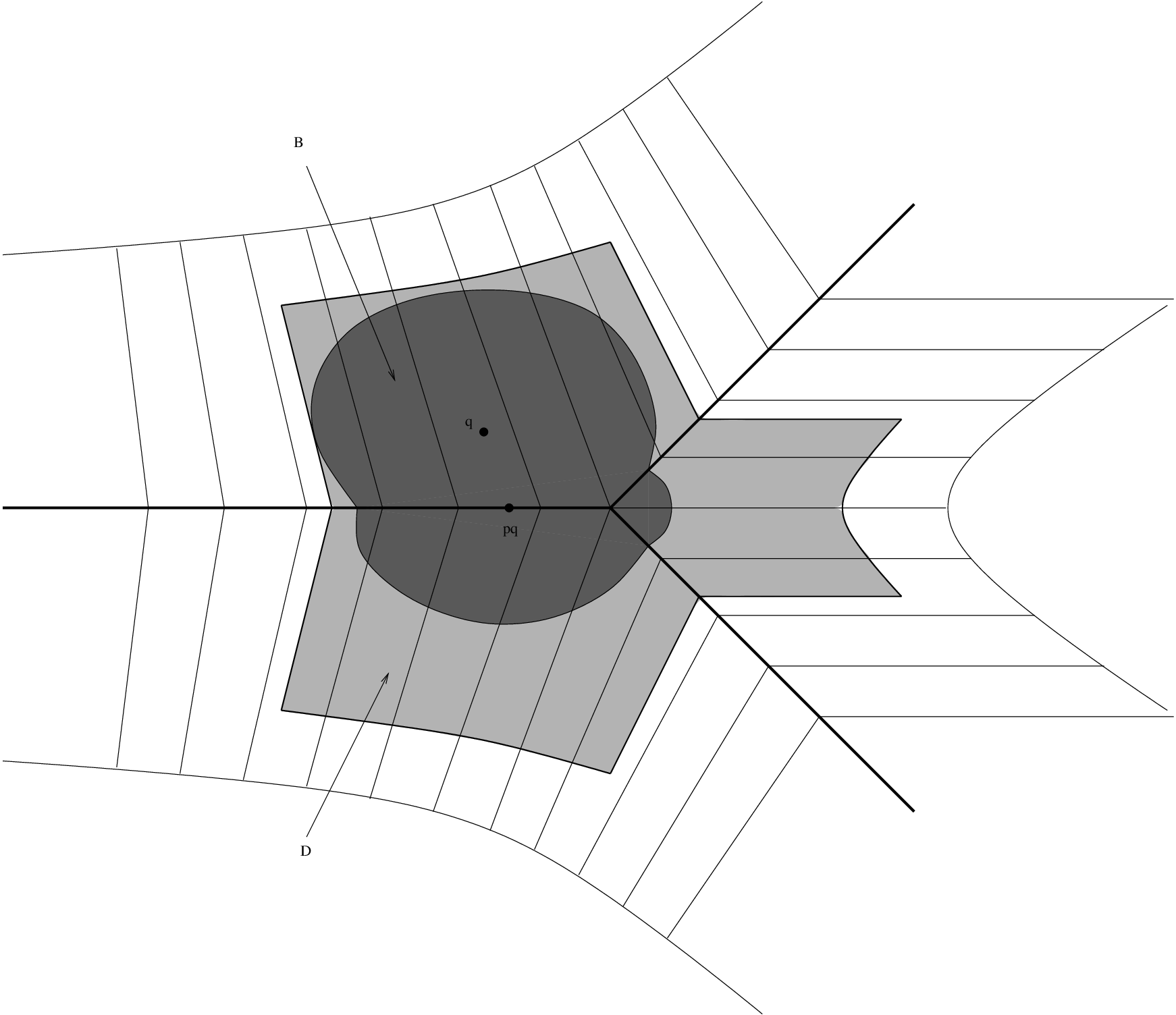}
\end{center}
\caption{The disk $D\left(\psi(q);\mu(q;r)\right)$ and the ball
$\ol{B}_S\qr$.}
\label{fig:ballindisk}
\end{figure}

In the next section, an explicit modulus of continuity $\rho_q\co
[0,\delta/2)\to[0,\infty)$ will be given for each~$q\in
    Q(\delta/2)$. For points~$q$ far from~$G$, a modulus of continuity
    is easily obtained by applying Theorem~\ref{thm:Koebedistn} to the
    reciprocal of~$\phi\co\Int(P)\to\csph$:

\begin{lem}
\label{lem:interior-mod-cont}
Let~$\tq,\tx\in P_{\delta/3} = P\setminus \tQ(\delta/3)$. Then
\[d_\csph(\phi(\tq),\phi(\tx))\le \kappa d_P(\tq,\tx),\]
where
\[\kappa = 2\exp\left(
\frac{24\diam_{P_{\delta/3}} P_{\delta/3}}
{\delta}
\right)
\le 2\exp\left(
\frac{48|\partial P|}{\delta}
\right).
\]
That is, $\phi$ is Lipschitz in~$P_{\delta/3}$ with a constant which
depends only on~$\br$, $\bh$, and $|\partial P|$.
\end{lem}
\qed

\subsection{Modulus of continuity at a point}
\label{sec:modul-cont-at-point}

\begin{defns}[$\xi$, $\lambda$, $\eta$, $\alpha$, $\beta$]
Define functions $\xi, \lambda, \eta\co Q(\delta/2)\times
[0,\delta/2)\to[0,\infty)$, and $\alpha,\beta\co Q(\delta/2)\to[0,\infty)$ by
\begin{eqnarray*}
\xi(q;t) &=& \max(h_q, t),\\
\lambda(q;t) &=& \min(\mu(q; \xi(q;t)),\,d_q),\\
\eta(q;t) &=& \max(\mu(q; \xi(q;t)),\,d_q),\\
\alpha(q) &=& \min(2d_q, \br), \text{ and }\\
\beta(q) &=& \max(2d_q, \br).
\end{eqnarray*}
\end{defns}

The aim of this section is to prove the following theorem, which gives
a modulus of continuity for~$u$ at each point~$q\in Q(\delta/2)$. The
expression~(\ref{eq:local-mod-cont}) giving this modulus of continuity
is difficult to parse: nevertheless it is amenable to explicit
calculation in examples such as Example~\ref{ex:q-singular} below.

\begin{thm}
\label{thm:local-modulus-continuity}

Let~$\rho\co Q(\delta/2)\times[0,\delta/2)\to[0,\infty)$ be defined by
    $\rho(q;t):=\rho_q(t)=0$ if $t=0$ and, if $t>0$,
\begin{equation}
\label{eq:local-mod-cont}
\rho_q(t) :=
\dfrac{8Rt}
{\displaystyle{\xi(q;t)\cdot\exp \left(
2\pi\int_{\lambda(q;t)}^{\alpha(q)/2}\iota(\psi(q);s)\,\rmd s +
2\pi\int_{d_q+\eta(q;t)}^{\beta(q)} \iota\SL(p;s)\,\rmd s
\right)}}
\end{equation}
where~$p$ is a point of~$\Lambda$ with $d_G(\psi(q),p)=d_q$, and~$R$
is the constant provided by Lemma~\ref{lem:Koebe}. Then, for
every~$q\in Q(\delta/2)$, $\rho_q$ is a modulus of continuity for
$u\co S\setminus\{p_\infty\}\to\C$ at~$q$ with respect to the
metric~$d_S$ on~$S$ and the Euclidean metric on~$\C$.
\end{thm}

\begin{rmk}
\label{rmk:mod-cont-makes-sense}
This expression makes sense. First, it does not depend on the
particular point~$p\in\Lambda$ with~$d_G(\psi(q),p)=d_q$. For
if~$p'\in\Lambda$ is a second such point, then $d_G(p,p')\le 2d_q$.
Therefore if $s>d_q$, as is the case throughout the range of the
second integral of~(\ref{eq:local-mod-cont}), then
$\CB(p;s)=\CB(p';s)$ by Remark~\ref{rmk:chains-in-Lambda}, and hence
$\CC(p;s)=\CC(p';s)$, $\cm(p;s)=\cm(p';s)$, $\cn(p;s)=\cn(p';s)$, and
finally $\iota\SL(p;s)=\iota\SL(p';s)$.

Second, if $\psi(q)\in\Lambda$ then $\iota(\psi(q);s)$ is not defined:
but in this case~$d_q=0$, so that the range of the first integral
of~(\ref{eq:local-mod-cont}) degenerates to a point.

Notice also that if~$q\in Q(\delta/2)$ and $d_S(q,x)<\delta/2$ then
$x\in Q(\delta)\sbs Q(\bh)$, and hence~$x\not=p_\infty\in S\setminus
Q(\bh)$. There is therefore no loss of generality in restricting the
domain of~$u$ to $S\setminus\{p_\infty\}$.
\end{rmk}

\begin{example}
\label{ex:q-singular}
If $q\in\Lambda$ then $h_q=d_q=0$ and $p=q$, so that $\xi(q;t)=t$,
$\lambda(q;t)=0$, $\eta(q;t)=\mu(q;t)=\br t/2\delta$, $\alpha(q)=0$, and
$\beta(q)=\br$. In this case~(\ref{eq:local-mod-cont}) therefore reduces
to
\[\rho_q(t) = \dfrac{8R}
{\displaystyle
{\exp\left(
2\pi\int_{\br t/2\delta}^{\br} \iota\SL(q;s)\,\rmd s
\right)}
}.
\]
Suppose, for example, that~$q$ is one of the points of the Cantor
singular set of Example~\ref{ex:cantor}. Taking~$\br=1/6$
and~$\bh=1/4$ and using~$|\partial P|=4$ gives~$\delta=1/192$, so that
the modulus of continuity is given by
\[\rho_q(t) = \dfrac{8R}
{\displaystyle
{\exp\left(
2\pi \int_{16t}^{1/6} \iota\SL(q;s)\,\rmd s
\right)}
}
\]
for $0\le t < 1/384$. Now, by the calculations of
Example~\ref{ex:cantor}, 
\[\iota\SL(q;s) \ge \frac{M}{\frac{1}{3^k}+7s} \qquad \text{when
}\quad\frac{1}{2\cdot 3^{k+2}} < s \le \frac{1}{2\cdot 3^{k+1}},\]
where~$M=\frac{1}{5}\min(\br/\bh,\bh/\br)=2/15$. This gives
\[\rho_q(t) \le C t^{\left(
\frac{4\pi\ln(39/25)}{105\ln 3}
\right)} \le C t^{1/21}
\]
as a modulus of continuity at any point~$q\in\cV^s$, for $0\le
t<1/384$. 
\end{example}

\begin{proof}[Proof of Theorem~\ref{thm:local-modulus-continuity}]

Let~$q\in Q(\delta/2)$. Observe first
that~$\rho_q\co[0,\delta/2)\to[0,\infty)$ is a modulus of continuity
    in the sense of Definition~\ref{defn:mod-cont}. For
\begin{itemize}
\item It is continuous in~$(0,\delta/2)$ since the functions~$t\mapsto
  \mu(q;t)$, $t\mapsto \xi(q;t)$, $t\mapsto \lambda(q;t)$, and
  $t\mapsto \eta(q;t)$ are all continuous, and $\xi(q;t)$ is non-zero
  in $t>0$.
\item It is strictly increasing since:
\begin{enumerate}[a)]
\item if $t<h_q$ then~$\xi(q;t)=h_q$, so that $\lambda(q;t)$ and
  $\eta(q;t)$ are independent of~$t$, and~$\rho_q(t)$ is
  proportional to~$t$; and
\item if $t\ge h_q$ then~$\xi(q;t)=t$ (cancelling the~$t$ in the
  numerator), so that $\mu(q; \xi(q;t))=\mu(q;t)$ is strictly
  increasing in~$t$. Hence $\lambda(q;t)$ and $\eta(q;t)$ are
  increasing, and one of them is strictly increasing. Since the
  integrands are positive except on a set of measure zero, the two
  integrals are decreasing and one of them is strictly decreasing.
\end{enumerate}
\item It is continuous at~$0$ since:
\begin{enumerate}[a)]
\item If~$h_q=0$ (i.e. $q\in G$) then $\xi(q;t)=t\to 0$ as $t\to 0$,
  so that $\mu(q; \xi(q;t))\to 0$ as $t\to 0$. If $d_q>0$ (so $q\in
  G\setminus\Lambda$ and $\alpha(q)>0$) then~$\lambda(q;t)\to 0$ as
  $t\to 0$, and the first integral in the denominator diverges
  as~$t\to 0$ by~(\ref{eq:regular-integral-diverges}). On the other
  hand, if~$d_q=0$ (so that $q\in\Lambda$) then $\eta(q;t)\to 0$ as
  $t\to 0$, and the second integral in the denominator diverges as
  $t\to 0$ by~(\ref{eq:generaldivint1}).
\item If~$h_q>0$ then $\rho_q(t)$ is proportional to~$t$
  for~$0<t<h_q$, so $\rho_q(t)\to 0$ as $t\to 0$.
\end{enumerate}
\end{itemize}

Now let~$x\in S$ with $d_S(q,x)=t<\delta/2$. It is required to show
that $|u(q)-u(x)|\le\rho_q(t)$.

The proof splits into three cases depending on whether $\psi(q)$ is
far from~$\Lambda$ (case A, $d_q\ge\br/2$); $\psi(q)$ is close
to~$\Lambda$ and $h_q$ is large compared to~$d_q$ (case B,
$d_q\le\br/2$ and $h_q\ge\frac{\delta}{\br}d_q$); or $\psi(q)$ is
close to~$\Lambda$ and $h_q$ is small compared to~$d_q$ (case C,
$d_q\le\br/2$ and $h_q\le\frac{\delta}{\br}d_q$). See
Figure~\ref{fig:cases}.

\begin{figure}[htbp]
\lab{0}{0}{}
\lab{1}{\br/2}{}
\lab{2}{d_q}{}
\lab{3}{\delta/2}{}
\lab{4}{h_q}{}
\lab{5}{h_q=\frac{\delta}{\br}d_q}{}
\lab{A}{A}{}
\lab{B}{B}{}
\lab{C}{C}{}
\begin{center}
\pichere{0.6}{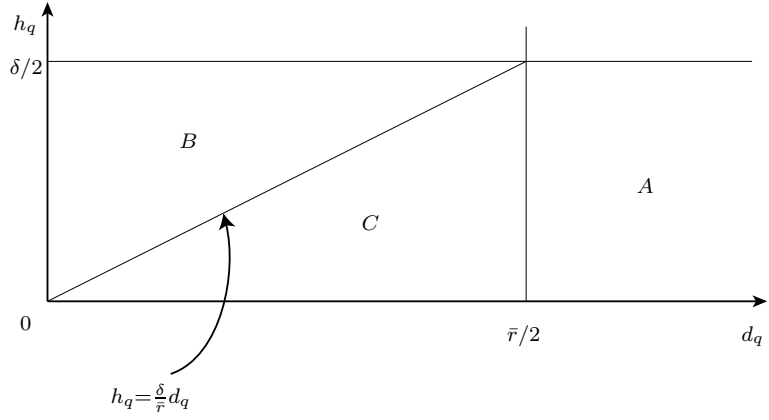}
\end{center}
\caption{The three cases in the proof of Theorem~\ref{thm:local-modulus-continuity}}
\label{fig:cases}
\end{figure}

Case~C is more complicated than the other two, and is treated in
detail. The proofs in cases~A and~B are similar, and are only
sketched. 

\medskip\medskip

\noindent{\bf Case C: }$d_q\le\br/2$ and $h_q\le
\frac{\delta}{\br}d_q$.

In this case~$\alpha(q)=2d_q$ and $\beta(q)=\br$. Notice that $0\le h_q \le
\frac{2\delta}{\br}d_q-h_q$. There are three subcases to consider:
$t\ge \frac{2\delta}{\br}d_q-h_q$, $h_q\le
t\le\frac{2\delta}{\br}d_q-h_q$, and $t\le h_q$.
\begin{enumerate}[(C1)]
\item Suppose first that~$\frac{2\delta}{\br}d_q-h_q \le
  t$. Then~$\xi(q;t)=t$, and
  $\mu(q;\xi(q;t))=\frac{\br}{2\delta}(t+h_q) \ge d_q$. Therefore
  $\lambda(q;t)=d_q$ (so that the first integral in the denominator
  of~(\ref{eq:local-mod-cont}) vanishes)
   and $\eta(q;t)=\mu(q;t)$. Thus~(\ref{eq:local-mod-cont}) reduces to
\[\rho_q(t) = \dfrac{8R}{\displaystyle
{\exp\left(
2\pi\int_{d_q+\mu(q;t)}^{\br}\iota\SL(p;s)\,\rmd s
\right)}
}.
\]

Assume to start with that both $d_q+\mu(q;t)$ and $\br$ are
$(p,\Lambda)$ planar radii. Now
\[x\in\ol{B}_S(q;t) \sbs \DL(p; d_q+\mu(q;t))\]
by Lemma~\ref{lem:technical}: for if $\psi(q) \not\in\Lambda$ then
$\ol{B}_S(q;t) \sbs D(\psi(q);\mu(q;t)) \sbs \DL(p; d_q+\mu(q;t))$;
while if $\psi(q)\in\Lambda$ then $\ol{B}_S(q;t) \sbs \DL(\psi(q);
\mu(q;t)) = \DL(p; d_q+\mu(q;t)$, since $d_q=0$ and
$\psi(q)=p$. Therefore both~$q$ and~$x$ lie in the bounded
complementary component of $\AnnL(p;d_q+\mu(q;t), \br)$ (i.e. the
complementary component which is contained in~$Q(\delta)$). See
Figure~\ref{fig:C1}.

\begin{figure}[htbp]
\lab{x}{x}{}
\lab{t}{t}{}
\lab{h}{h_q}{}
\lab{k}{\psi(q)}{}
\lab{q}{q}{}
\lab{p}{p}{}
\lab{d}{d_q}{}
\lab{A}{\AnnL(p;\, d_q+\mu(q;t), \br)}{}
\lab{B}{\ol{B}_S(q;t)}{}
\begin{center}
\pichere{0.65}{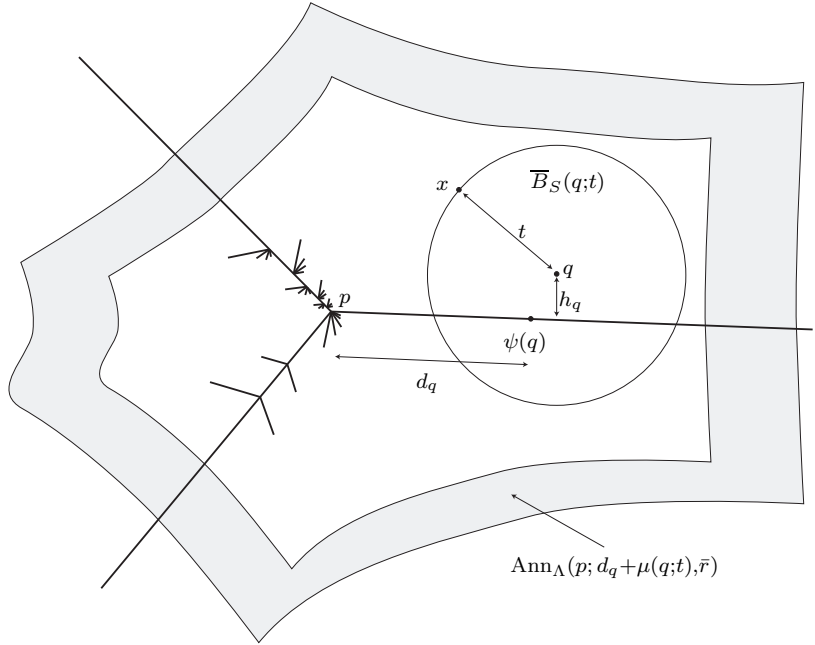}
\end{center}
\caption{The annulus of case~(C1)}
\label{fig:C1}
\end{figure}

By Lemma~\ref{lem:Koebe}, the image $u(\AnnL(p;d_q+\mu(q;t), \br))$
separates~$u(q)$ and~$u(x)$ from the circle $|z|=R$. Applying
Theorem~\ref{thm:module-bound} and Theorem~\ref{thm:GAT} gives
\begin{eqnarray*}
\int_{d_q+\mu(q;t)}^{\br}\iota\SL(p;s)\,\rmd s &\le&
\mod\AnnL(p;d_q+\mu(q;t),\br) \\
&=& \mod u(\AnnL(p;d_q+\mu(q;t),\br)) \\
&\le& \frac{1}{2\pi} \ln \frac{8R}{|u(q)-u(x)|},
\end{eqnarray*}
from which the required inequality $|u(q)-u(x)|\le\rho_q(t)$ follows.

If either or both of $d_q+\mu(q;t)$ and $\br$ are not
$(p,\Lambda)$-planar, then increase $d_q+\mu(q;t)$ and decrease $\br$
by arbitrarily small amounts to $(p,\Lambda)$-planar radii. This
increases the upper bound on~$|u(q)-u(x)|$, but since it does so by an
arbitrarily small amount, the upper bound as stated remains valid. In
the remaining cases of the proof, this finessing of the possibility
that relevant radii are not planar will be carried out without
comment. 

\item If $h_q\le t\le \frac{2\delta}{\br}d_q-h_q$ then $\xi(q;t)=t$
  and $\mu(q;\xi(q;t))=\frac{\br}{2\delta}(t+h_q) \le d_q$. Therefore
  $\lambda(q;t)=\mu(q;t)$ and $\eta(q;t)=d_q$, so
  that~(\ref{eq:local-mod-cont}) becomes
\[\rho_q(t) = \dfrac{8R}{\displaystyle
{
\exp\left(
2\pi\int_{\mu(q;t)}^{d_q}\iota(\psi(q);s)\,\rmd s + 
2\pi\int_{2d_q}^{\br}\iota\SL(p;s)\,\rmd s
\right)
}
}.
\]
By Lemma~\ref{lem:technical}, $x\in \ol{B}_S(q;t) \sbs
D(\psi(q);\mu(q;t))$ so that $\Ann(\psi(q);\mu(q;t), d_q)$
separates~$q$ and~$x$ from $\partial Q(\delta)$. This annulus is
itself nested in the annulus $\AnnL(p; 2d_q, \br)$, since $D(\psi(q);
d_q) \sbs \DL(p; 2d_q)$: see Figure~\ref{fig:C2}

\begin{figure}[htbp]
\lab{x}{x}{}
\lab{t}{t}{}
\lab{k}{\psi(q)}{}
\lab{q}{q}{}
\lab{p}{p}{}
\lab{A}{\AnnL(p;\, 2d_q, \br)}{}
\lab{B}{\ol{B}_S(q;t)}{}
\lab{C}{\Ann(\psi(q);\,\mu(q;t), d_q)}{}
\begin{center}
\pichere{0.65}{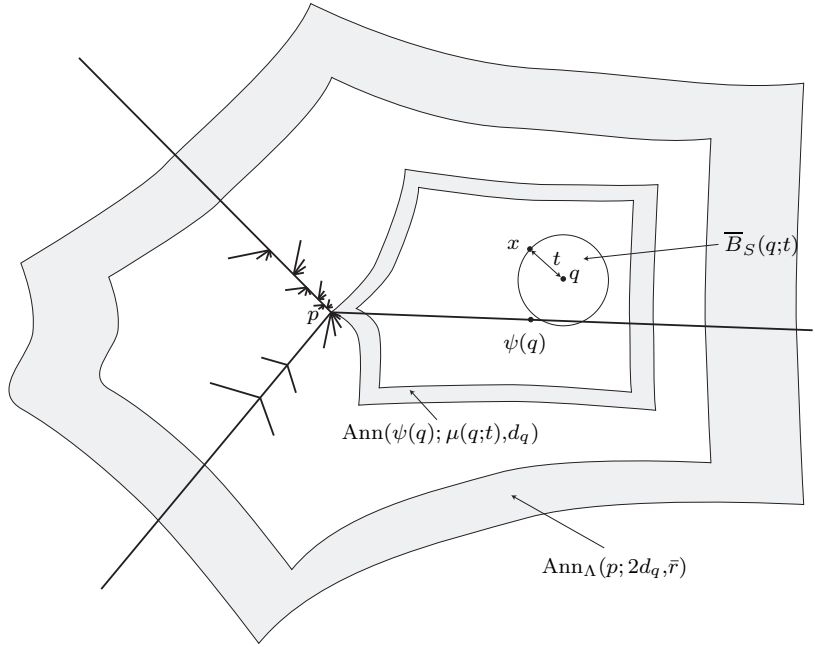}
\end{center}
\caption{The two annuli of case~(C2)}
\label{fig:C2}
\end{figure}

Therefore the annulus~$A$ with inner
boundary~$\partial D(\psi(q);\mu(q;t))$ and outer boundary
$\partial\DL(p;\br)$ satisfies (using
(\ref{eq:standard-mod-lower-bound}) and
Theorems~\ref{thm:module-bound} and~\ref{thm:GAT})
\begin{eqnarray*}
\int_{\mu(q;t)}^{d_q}\iota(\psi(q);s)\,\rmd s + 
\int_{2d_q}^{\br}\iota\SL(p;s)\,\rmd s &\le& \mod A \\
&=& \mod u(A) \\
&\le&
\frac{1}{2\pi}\ln\frac{8R}{|u(q)-u(x)|},
\end{eqnarray*}
and the required inequality follows.
\item Finally, if $0<t\le h_q$ then $\xi(q;t)=h_q$,
  $\lambda(q;t)=\mu(q;h_q)$ and
  $\eta(q;r)=d_q$. Therefore~(\ref{eq:local-mod-cont}) becomes
\[\rho_q(t) = \dfrac{8Rt}{\displaystyle
{
h_q\cdot \exp\left(
2\pi\int_{\mu(q;h_q)}^{d_q}\iota(\psi(q);s)\,\rmd s + 
2\pi\int_{2d_q}^{\br}\iota\SL(p;s)\,\rmd s
\right)
}
}.
\]
By Lemma~\ref{lem:technical}, $\ol{B}_S(q;h_q) \sbs D(\psi(q);
\mu(q;h_q))$, so the ``round'' annulus
$A=B_S(q;h_q)\setminus \ol{B}_S(q;t)$ is nested inside
$\Ann(\psi(q);\mu(q;h_q), d_q)$, which in turn is nested inside
$\AnnL(p; 2d_q, \br)$: see Figure~\ref{fig:C3}.

\begin{figure}[htbp]
\lab{x}{x}{}
\lab{k}{\psi(q)}{}
\lab{q}{q}{}
\lab{p}{p}{}
\lab{A}{\AnnL(p;\, 2d_q, \br)}{}
\lab{B}{B_S(q;\,h_q)\setminus \ol{B}_S(q;t)}{}
\lab{C}{\Ann(\psi(q);\,\mu(q;h_q), d_q)}{}
\begin{center}
\pichere{0.65}{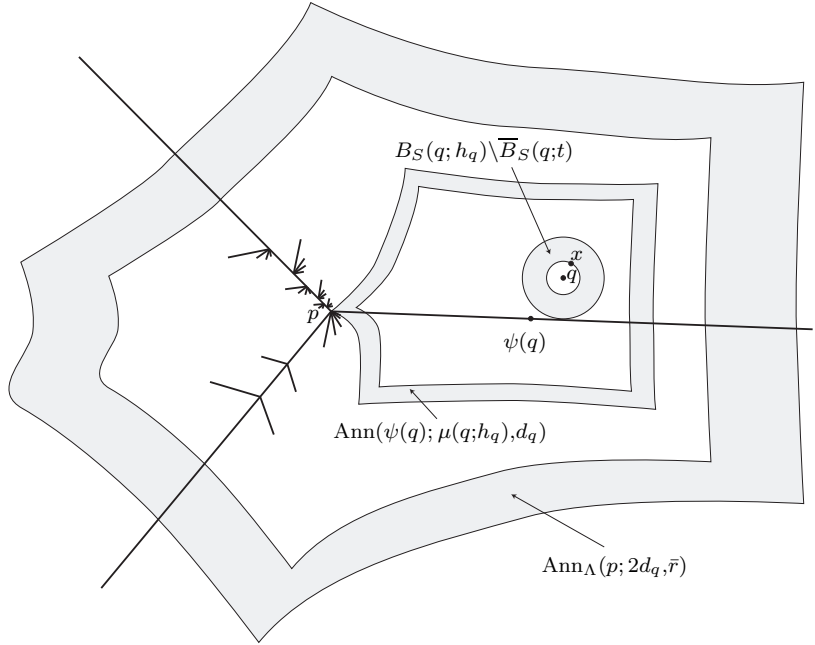}
\end{center}
\caption{The three annuli of case~(C3)}
\label{fig:C3}
\end{figure}

Since~$A$ is contained entirely in~$S\setminus
G$ and is therefore conformally equivalent to the standard annulus
$A(t,h_q)$, it has modulus
$\frac{1}{2\pi}\ln\frac{h_q}{t}$. Applying~(\ref{eq:standard-mod-lower-bound})
and Theorems~\ref{thm:module-bound} and~\ref{thm:GAT} therefore gives
\[\frac{1}{2\pi}\ln\frac{h_q}{t} + \int_{\mu(q;h_q)}^{d_q}\iota(\psi(q);s)\,\rmd s + 
\int_{2d_q}^{\br}\iota\SL(p;s)\,\rmd s \le
\frac{1}{2\pi}\ln\frac{8R}{|u(q)-u(x)|},\]
and the required inequality follows.

\medskip\medskip

\noindent{\bf Case A: }$d_q\ge\br/2$.

In this case $\mu(q;t)=\frac{\br}{2\delta}(t+h_q) \le
\frac{\br}{2}\le d_q$ for all $(q;t)\in
Q(\delta/2)\times[0,\delta/2)$, so that $\lambda(q;t)=\mu(q; \xi(q;t))$
  and $\eta(q;t)=d_q$. Moreover~$\alpha(q)=\br$ and~$\beta(q)=2d_q$: in
  particular, the second integral in the denominator
  of~(\ref{eq:local-mod-cont}) vanishes.

There are two subcases to consider. 

\begin{enumerate}[({A}1)]
\item If $h_q\le t$ then the result comes from
  considering the annulus $\Ann(\psi(q);\mu(q,t), \br/2)$.
\item If $0<t\le h_q$ then the result comes from considering the
  annulus $\Ann(\psi(q); \mu(q;h_q), \br/2)$ together with the
  ``round'' annulus $B_S(q;h_q)\setminus\ol{B}(q;t)$, which is nested
  inside it.
\end{enumerate}

\medskip\medskip

\noindent{\bf Case B: }$d_q\le\br/2$ and $h_q\ge
\frac{\delta}{\br}d_q$. 

In this case $\mu(q;\xi(q;t))\ge \mu(q;h_q)=\frac{\br}{\delta}h_q\ge
d_q$, so that $\lambda(q;t)=d_q$ and
$\eta(q;t)=\mu(q;\xi(q;t))$. Moreover~$\alpha(q)=2d_q$ and $\beta(q)=\br$: in
particular, the first integral in the denominator
of~(\ref{eq:local-mod-cont}) vanishes.

There are two subcases to consider.

\begin{enumerate}[(B1)]
\item If $h_q\le t$ then the result comes from considering the annulus
  $\AnnL(p;d_q+\mu(q;t), \br)$ (the expression for~$\rho_q(t)$ in this
  case is identical to that of case~(C1)).
\item If $0<t\le h_q$ then the result comes from considering the
  annulus $\AnnL(p; d_q+\mu(q;h_q), \br)$ together with the ``round''
  annulus $B_S(q;h_q)\setminus \ol{B}(q;t)$ which is nested inside it.
\end{enumerate}

\end{enumerate}

\end{proof}

\subsection{Global modulus of continuity}
\label{sec:glob-modul-cont}

In order to construct a global modulus of continuity from the local
moduli of continuity of the previous section, it is first necessary to
prove that the function $\rho$ is (jointly) continuous on its domain
$Q(\delta/2)\times[0,\delta/2)$, so that $\max_{q\in
    Q(\delta/2)}\rho(q,t)$ is a modulus of continuity
  throughout~$Q(\delta/2)$. Lemma~\ref{lem:cty-basis} provides the
  basis for this argument, which is completed in
  Lemma~\ref{lem:rho-continuous}. This is then combined with the
  modulus of continuity in~$P_{\delta/3}$ provided by
  Lemma~\ref{lem:interior-mod-cont} to give the required global
  modulus (Theorem~\ref{thm:mod-cont}).

\begin{lem}
\label{lem:cty-basis}
Let~$I_1\co G\setminus\Lambda\times(0,\br)^2 \to\R$ and
$I_2\co\Lambda\times(0,\br)\to(0,\infty)$ be defined by
\begin{eqnarray*}
I_1\qrt &=&\int_r^t\iota\qs\,\rmd s, \quad\text{ and} \\
I_2\qr &=& \int_r^{\br} \iota\SL\qs\,\rmd s.
\end{eqnarray*}
Then
\begin{enumerate}[a)]
\item $I_1$ is continuous.
\item For all~$K>0$, $q_0\in G\setminus\Lambda$, and $t_0\in(0,\br)$,
  there is some~$\eta>0$ such that if~$q\in B_G(q_0;\eta)$ and
  $r\in(0,\eta)$, then $I_1(q;r,t_0)>K$.
\item For all~$K>0$ and $q_0\in\Lambda$, there is some~$\eta>0$ such
  that if~$q\in\Lambda\cap B_G(q_0;\eta)$ and $r\in(0,\eta)$, then
  $I_2\qr>K$. 
\end{enumerate}

\end{lem}

\begin{proof}
Recall that 
\[\iota\qs = \frac{M}{m(q;s)+s\cdot n(q;s)} \quad\text{ and } \quad
\iota\SL\qs = \frac{M}{\cm\qs+ s\cdot\cn\qs}.\]
By Remark~\ref{rmk:upper-bound}, $\iota\SL\qs$ is bounded above by
$M/2s$, and similarly $\iota\qs$ is bounded above by~$M/2s$.
\begin{enumerate}[a)]
\item Let~$(q_0;r_0,t_0)\in G\setminus\Lambda\times(0,\br)^2$: it will be shown
  that~$I_1$ is continuous at~$(q_0;r_0,t_0)$. Now for $\qrt\in
  G\setminus\Lambda\times(0,\br)^2$, 
\[I\qrt - I(q_0;r_0,t_0) = \int_{r_0}^{t_0}\iota(q;s)\,\rmd s -
\int_{r_0}^{t_0}\iota(q_0;s)\,\rmd s  + \int_{r}^{r_0}\iota(q;s)\,\rmd
s + \int_{t_0}^{t}\iota(q;s)\,\rmd s.\]
The final two integrals converge to zero as $(q;r,t)\to(q_0;r_0,t_0)$,
since~$\iota(q;s)$ is bounded above by $M/r_0$ for $s\ge r_0/2$. Hence
it suffices to prove that, for any~$\veps>0$,
\[
\left|\int_{r_0}^{t_0}\left(
\iota\qs - \iota(q_0;s)
\right)\,\rmd s \right|<\veps
\]
provided that $d_G(q,q_0)$ is sufficiently small.

Let~$B:=[r_0,t_0]\cap \{d_G(q_0,q^*)\,:\,q^*\in\ol\cV\}$. Then (cf. the
proof of Lemma~\ref{lem:generalplanarradius}), $B$ is a compact set
with zero Lebesgue measure, and hence can be covered by a finite
union~$U$ of open intervals with total length less than $\veps
r_0/2M$. Let~$L:=[r_0,t_0]\setminus U$ and let $\delta>0$ be the
minimum distance from a point of~$L$ to a point of~$B$. Then
\[\int_L\iota(q_0;s)\,\rmd s \ge \int_{r_0}^{t_0}\iota(q_0;s)\,\rmd
s - \frac{\veps r_0}{2M} \cdot \frac{M}{2r_0} =
\int_{r_0}^{t_0}\iota(q_0;s)\,\rmd s - \frac{\veps}{4},\]
\[\int_L\iota(q;s)\,\rmd s \ge \int_{r_0}^{t_0}\iota(q;s)\,\rmd s -
\frac{\veps}{4},\] 
and $(d_G(q_0,q^*)-\delta, d_G(q_0,q^*)+\delta)$ is disjoint from~$L$
for all $q^*\in\ol\cV$.

It follows that $\ol{B}_G(q;b)\setminus B_G(q;a)$ is a disjoint union
of $n(q_0;a)$ intervals of length~$b-a$ for each component~$[a,b]$
of~$L$ and each~$q\in B_G(q_0;\delta)$, so that $n(q;s)=n(q_0;a)$ is
constant for $q\in B_G(q_0;\delta)$ and $s\in[a,b]$, and $m(q;s)$ is
continuous on the domain $B_G(q_0;\delta)\times[a,b]$. Therefore if
$d_G(q,q_0)$ is sufficiently small then
\[
\left|
\int_L\iota(q_0;s)\,\rmd s - \int_L\iota(q;s)\,\rmd s
\right| \le \frac{\veps}{2},
\]
and the result follows.

\item Let~$q_0\in G\setminus\Lambda$, $K>0$, and
  $t_0\in(0,\br)$. $\eta\in(0,t_0)$ will be chosen small enough that
  $B_G(q_0;\eta)$ is disjoint
  from~$\Lambda$. Since~(\ref{eq:regular-integral-diverges}) holds
  at~$q_0$, there is some~$\veps>0$ such that
\[\int_\veps^{t_0}\iota(q_0;s)\,\rmd s > 3K.\]
As in part~a), there is a subset~$L$ of~$[\veps,t_0]$ consisting of a
finite union of closed intervals, and a number~$\delta>0$ such that
\[\int_L\iota(q_0;s)\,\rmd s > 2K\]
and $(d_G(q_0,q^*)-\delta, d_G(q_0,q^*)+\delta)$ is disjoint from~$L$
for all $q^*\in\ol\cV$.

Again as in part~a), there is some~$\delta'\in(0,\delta)$ such that
\[\int_L\iota(q;s)\,\rmd s > K\]
provided that~$d_G(q;q_0)<\delta'$. Hence
\[I_1(q;r,t_0) > \int_{\veps}^{t_0}\iota\qs\,\rmd s \ge
\int_L\iota\qs\,\rmd s > K\]
for all~$q$ with $d_G(q,q_0)<\delta'$ and all $r\in(0,\veps)$, as
required. 

\item Let~$q_0\in\Lambda$ and $K>0$. Since~(\ref{eq:generaldivint1})
  holds at~$q_0$, there is some~$\veps>0$ such that
\[\int_\veps^{\br} \iota\SL(q_0;s)\,\rmd s > K.\]
Now if $q\in\Lambda$ with $d_G(q_0,q)<2\veps$ then
$\iota\SL(q;s)=\iota\SL(q_0;s)$ for all~$s\ge\veps$ by
Remark~\ref{rmk:chains-in-Lambda}. Hence if $d_G(q_0,q)<2\veps$ and
$r\in(0,\veps)$ then
\[I_2\qr > \int_\veps^\br \iota\SL\qs\,\rmd s  = \int_\veps^\br
\iota\SL(q_0;s) \,\rmd s > K\]
as required.

\end{enumerate}

\end{proof}

\begin{lem}
\label{lem:rho-continuous}
The function $\rho\co Q(\delta/2)\times[0,\delta/2) \to
   [0,\infty)$ of~(\ref{eq:local-mod-cont}) is continuous.
\end{lem}

\begin{proof}
Let~$(q_0;t_0)\in Q(\delta/2)\times[0,\delta/2)$: it will be shown
  that~$\rho$ is continuous at~$(q_0;t_0)$.

\medskip\medskip

\noindent\textbf{Case 1: }$t_0>0$.

Notice that in this case $\xi(q;t)>t_0/2$, and hence $\mu(q;\xi(q;t))>
\br t_0/4\delta$ for all~$(q;t)$ sufficiently close to~$(q_0;t_0)$. 
\begin{enumerate}[a)]
\item
Suppose first that $\psi(q_0)\in\Lambda$, so that $d_{q_0}=0$. Then
$\lambda(q;t)=d_q$, $\eta(q;t)=\mu(q;\xi(q;t))$, $\alpha(q)=2d_q$, and
$\beta(q)=\br$ for all~$(q;t)$ sufficiently close to~$(q_0;t_0)$, so that
\[\rho(q;t)=
\dfrac{8Rt}
{\displaystyle
{\xi(q;t)\cdot \exp\left(
2\pi\int_{d_q+\mu(q;\xi(q;t))}^\br \iota\SL(p;s)\,\rmd s
\right)}
},
\]
where~$p\in\Lambda$ satisfies $d_G(p, \psi(q))=d_q$. However if~$q$ is
sufficiently close to~$q_0$ then $\iota\SL(p;s)=\iota\SL(\psi(q_0);s)$
throughout the range of integration, and the continuity of~$\rho$
follows from the continuity of $\mu$, $\xi$, and $d_q$.
\item
Suppose that~$\psi(q_0)\not\in\Lambda$, so
that~$d_{q_0}>0$. Let~$p_0\in\Lambda$ with
$d_G(p_0,q_0)=d_{q_0}$. Pick~$\veps<d_{q_0}$, and suppose that $q\in
Q(\delta/2)$ with $d_G(\psi(q),\psi(q_0))<\veps$: thus
$|d_{q_0}-d_q|<\veps$. Let~$p\in\Lambda$ with $d_G(p,q)=d_q$. Then
$d_G(p,p_0) \le 2(d_{q_0}+\veps)$. Since $s\ge 2(d_{q_0}-\veps)$
throughout the range of integration of the second integral
of~(\ref{eq:local-mod-cont}), $\iota\SL(p_0;s)=\iota\SL(p;s)$
throughout the range of integration for $\veps$ sufficiently small by
Remark~\ref{rmk:chains-in-Lambda}, and this second integral varies
continuously with~$(q;t)$.  The first integral is~$I_1(\psi(q);
\lambda(q;t), \alpha(q)/2)$, which is continuous at~$(q_0;t_0)$ by
Lemma~\ref{lem:cty-basis}~a), since $\lambda(q_0;t_0)>0$.
\end{enumerate}

\medskip\medskip

\noindent\textbf{Case 2: }$t_0=0$.

In this case it is necessary to show that $\rho(q;t)\to 0$ as
$(q;t)\to(q_0;0)$ with~$t>0$. 
\begin{enumerate}[a)]
\item If~$h(q_0)>0$ then $\xi(q;t)=h_q$, and hence $\rho(q;t)\le
  8Rt/h_q$, for all~$(q;t)$ sufficiently close to $(q_0;0)$, and the
  result follows.
\item If~$h(q_0)=0$ and $d_{q_0}>0$ (i.e. $q_0\in G\setminus\Lambda$),
  then $\alpha(q_0)>0$ and, ignoring the second integral in the denominator
  of~(\ref{eq:local-mod-cont}), 
\[\rho(q;t)\le\dfrac
{8R}
{\displaystyle
{\exp\left(2\pi I_1(\psi(q);\lambda(q;t),\alpha(q_0)/4)\right)}
}
\]
for~$(q;t)$ sufficiently close to~$(q_0;0)$, and the result follows by
Lemma~\ref{lem:cty-basis}~b) since $\lambda(q;t)\to 0$ as
$(q;t)\to(q_0;0)$. 
\item If~$h(q_0)=0$ and $d_{q_0}=0$ (i.e. $q_0\in\Lambda$),
  then
\[\rho(q;t)\le\dfrac
{8R}
{\displaystyle
{\exp\left(2\pi I_2(p; d_q+\eta(q;t))\right)}
}
\]
for~$(q;t)$ sufficiently close to~$(q_0;0)$, and the result follows by
Lemma~\ref{lem:cty-basis}~c) since $d_q+\eta(q;t)\to 0$ and $p\to q_0$
as $(q;t)\to(q_0;0)$.
\end{enumerate}
\end{proof}

\begin{thm}
\label{thm:mod-cont}
Let~$(P,\cP)$ be a plain paper-folding scheme and $\Lambda=\cV^s$ be
the singular set. Let~$\br$ be an injectivity radius for~$\Lambda$ and
$\bh$ be a collaring height for~$P$. Suppose that
condition~(\ref{eq:generaldivint1}) holds at every
point~$q\in\Lambda$. Then the uniformizing map $\phi=u\circ\pi\co
P\to\csph$ has a modulus of continuity~$\brho$, with respect to the
Euclidean metric on~$P$ and the spherical metric on~$\csph$, which
depends only on $\br$, $\bh$, $|\partial P|$, and the functions
$\iota\SL\co\Lambda \times(0,\br)\to[0,\infty)$ and $\iota\co
  G\setminus\Lambda\times(0,\br) \to[0,\infty)$.
\end{thm}

\begin{proof}
Define~$\hrho\co[0,\delta/2)\to [0,\infty)$ by
\[\hrho(t):=2\max_{q\in Q(\delta/2)} \rho(q,t),\]
which is well defined since~$\rho$ is continuous and~$Q(\delta/2)$ is
compact. Then~$\hrho$ is a continuous strictly increasing function
with $\hrho(0)=0$ (since each~$\rho_q$ has these properties and~$\rho$
is continuous), and $\phi$ has modulus of continuity~$\hrho$ on
$\tQ(\delta/2)$ with respect to the Euclidean metric on~$\tQ(\delta)$
and the spherical metric on~$\csph$ (the factor~2 in the definition
of~$\hrho$ arises from the translation from the Euclidean metric
on~$\C$ to the spherical metric on~$\csph\setminus\{\infty\}$). 

On the other hand, $\phi$ is~$\kappa$-Lipschitz in~$P_{\delta/3}$ by
Lemma~\ref{lem:interior-mod-cont}. Hence
$\brho\co[0,\delta)\to[0,\infty)$ defined by
\[\brho(t):=\max\{\hrho(t),\,\kappa t\}\]
is the desired modulus of continuity.
\end{proof}

\bibliography{bib_origami}

\end{document}